\documentclass[12pt]{amsart}
\usepackage[utf8]{inputenc}

\usepackage{mathrsfs}
\usepackage{amssymb}
\usepackage{bm}
\usepackage{graphicx}
\usepackage[centertags]{amsmath}
\usepackage{amsfonts}
\usepackage{amsthm}
\usepackage{enumerate}
\usepackage{tocvsec2}
\usepackage{xcolor}
\usepackage[margin=1in]{geometry}

\newtheorem{theorem}{Theorem}[section]
\newtheorem*{theorem*}{Theorem}

\newtheorem{corollary}[theorem]{Corollary}
\newtheorem{lemma}[theorem]{Lemma}
\newtheorem{rem}[theorem]{Remark}

\newtheorem{proposition}[theorem]{Proposition}

\theoremstyle{definition}

\newenvironment{remark}[1][Remark]{\begin{trivlist}
\item[\hskip \labelsep {\bfseries #1}]}{\end{trivlist}}

\newcommand{\ord}{\textbf{Ord}}

\newcommand{\nn}{\mathbb{N}}

\newcommand{\ee}{\varepsilon}

\newcommand{\eee}{\mathcal{E}}

\newcommand{\ttt}{\mathcal{M}}

\newcommand{\uuu}{\mathcal{U}}

\newcommand{\cat}{^\smallfrown}

\begin{document}

\title[An ordinal index characterizing weak compactness]{An ordinal index characterizing \\ weak compactness of operators}
\author{RM Causey}

\begin{abstract} We introduce an ordinal index which characterizes weak compactness of operators between Banach spaces.  We study when classes consisting of operators having bounded index form a closed ideal, the distinctness of the classes, and the descriptive set theoretic properties of this index. 

\end{abstract}

\maketitle

\section{Introduction}

In this work, we introduce an ordinal index for operators on Banach spaces in order to characterize weak compactness.  The construction of such an index was inspired by the Szlenk index, which is used to characterize Asplund operators \cite{Brooker}.  To define our index, we use a characterization of weakly compact operators due to James concerning sequences in the unit ball of the domain such that the images under a given operator have a certain convex separation property. Consequently, we denote our index by $\mathcal{J}$.   For a weakly compact operator $A:X\to Y$ between Banach spaces, we denote the James index of $A$ by $\mathcal{J}(A)$.  For an ordinal $\xi$, we let $\mathscr{J}_\xi$ denote the class consisting of all weakly compact operators $A$ with $\mathcal{J}(A)\leqslant \xi$. We write $\mathcal{J}(A)=\infty$ whenever $A$ fails to be weakly compact, with the precise meaning of this made clear below.    In this work, we establish the following results.  

\begin{theorem} For each ordinal $\xi$, $\mathscr{J}_{\omega^{\omega^\xi}}$ is a closed operator ideal.  The class $\mathscr{J}_\omega$ consists of all super weakly compact operators.  The $\mathcal{J}$ index of an operator on a separable domain is countable if and only if that operator is weakly compact. For every ordinal $\xi$, there exists a weakly compact operator which does not lie in $\mathscr{J}_\xi$, and that operator can be taken to have a separable domain if $\xi$ is countable.  The index $\mathcal{J}$ is a coanalytic rank on the class of operators between separable Banach spaces.  

\end{theorem}

\section{The index}

In this work, all Banach spaces are assumed to be real, while the modifcations for the complex case are straightforward.   By ``subspace,'' we shall mean a closed subspace.  By ``operator,'' we shall mean a bounded, linear operator between Banach spaces.  We let $B_X$ denote the closed unit ball of $X$.  We let $2=\{0,1\}$.  

Given a set $S$, we let $S^{<\nn}$ denote the finite sequences in $S$, including the empty sequence, denoted $\varnothing$.  We let $S^\nn$ denote the infinite sequences in $S$.  Given $t\in S^{<\nn}\cup S^\nn$, we let $|t|$ denote the length of $t$, and for an integer $i$ with $0\leqslant i\leqslant |t|$, we let $t|_i$ denote the initial segment of $t$ having length $i$.  We let $t|^i$ denote the tail of $t$ which remains after removing $t|_i$ from $t$.  Given $s,t\in S^{<\nn}$, we let $s\cat t$ denote the concatenation of $s$ with $t$, listing the members of $s$ first.  

We define the order $\preceq$ on $S^{<\nn}$ by letting $s\preceq t$ if and only if $|s|\leqslant |t|$ and $s=t|_{|s|}$. That is, if and only if $s$ is an initial segment of $t$.  We say a subset $T\subset S^{<\nn}$ is a \emph{tree} provided it is downward closed with respect to the order $\preceq$.  We let $MAX(T)$ denote the maximal members of $T$.  That is, $MAX(T)$ consists of those members of $T$ which do not have a proper extension in $T$.  We define $T'=T\setminus MAX(T)$, and call $T'$ the \emph{derived tree} of $T$.  Note that $T'$ is also a tree.  We then define the transfinite derived trees $T^\xi$ by $T^0=T$, $T^{\xi+1}=(T^\xi)'$, and $T^\xi=\cap_{\zeta<\xi}T^\zeta$ when $\xi$ is a limit ordinal. Note that there exists an ordinal $\xi$ so that $T^\xi=T^{\xi+1}=T^\zeta$ for all $\zeta>\xi$.   We say $T$ is \emph{ill-founded} if there exists an infinite sequence $(x_i)$ in $S$ so that $(x_i)_{i=1}^n\in T$ for all $n\in \nn$, and $T$ is \emph{well-founded} otherwise.  Note that $T$ is well-founded if and only if whenever $T^\xi=T^{\xi+1}$, $T^\xi=\varnothing$.   In the case that $T$ is well-founded, we let $o(T)$ denote the minimum ordinal $\xi$, called the \emph{order} of $T$, such that $T^\xi=T^{\xi+1}=\varnothing$.  If $T$ is ill-founded, we write $o(T)=\infty$.  We establish the convention that $\xi<\infty$ for all ordinals $\xi$.  

Suppose $X$ is a Banach space.  Suppose that $K\subset X^*$ is non-empty, symmetric, convex, and $w^*$ compact (such sets will henceforth be called \emph{bodies}).  Define the seminorm $|\cdot|_K$ on $X$ by $|x|_K=\sup_{x^*\in K}x^*(x)$.  Given two non-empty subsets $S_1$, $S_2$ of $S$, we let $d_K(S_1, S_2)=\inf_{x\in S_1, y\in S_2}|x-y|_K$.  For $\ee>0$, we say a sequence $t\in X^{<\nn}$ is $(K, \ee)$-\emph{cs} (for \emph{convexly separated}) if for any $1\leqslant m< |t|$, $d_K(\text{co}(t|_m), \text{co}(t|^m))\geqslant \ee$.  We consider the empty sequence to be $(K, \ee)$-cs for every $K$ and $\ee$.  We say that an infinite sequence is $(K, \ee)$-cs if all of its initial segments are $(K, \ee)$-cs.  

Note that by the Hahn-Banach theorem, the condition that the sequence $t=(x_i)_{i=1}^{|t|}$ is $(K, \ee)$-cs is equivalent to the following: For every $1\leqslant m< |t|$, there exists $x^*\in K$ so that for each $1\leqslant i\leqslant m<j\leqslant |t|$, $x^*(x_i-x_j)\geqslant \ee$.  Note also that if $A:X\to Y$ is an operator, the condition that $(x_i)_{i=1}^n$ is $(A^*B_{Y^*}, \ee)$-cs is equivalent to the condition that $(Ax_i)_{i=1}^n$ is $(B_{Y^*}, \ee)$-cs.  That is, for any $1\leqslant m<n$ and non-negative scalars $(a_i)_{i=1}^n$ with $1=\sum_{i=1}^m a_i=\sum_{i=m+1}^n a_i$, $\|\sum_{i=1}^m a_i Ax_i- \sum_{i=m+1}^n a_i Ax_i\|\geqslant \ee$.  

Given a Banach space $X$, a body $K\subset X^*$, and $\ee>0$, we let $J(K, \ee)$ denote the tree consisting of all $(K, \ee)$-cs sequences in $B_X^{<\nn}$.  We let $j(K, \ee)$ denote the order of $J(K, \ee)$. We define $\mathcal{J}(K)=\sup_{\ee>0} j(K, \ee)$.  If $A:X\to Y$ is an operator, we let $\mathcal{J}(A)=\mathcal{J}(A^*B_{Y^*})$.  For an ordinal $\xi$, we let $\mathscr{J}_\xi$ be the class of all operators $A$ so that $\mathcal{J}(A)\leqslant \xi$.  The main result of this work is the following. 

\begin{theorem} For every $\xi\in \emph{\ord}$, $\mathscr{J}_{\omega^{\omega^\xi}}$ is a closed operator ideal.  Moreover, $\mathscr{J}_\omega$ is the ideal of all super weakly compact operators and  $\cup_{\xi\in \emph{\ord}}\mathscr{J}_\xi$ is the ideal of weakly compact operators. For all $\zeta\in \emph{\ord}$, there exists $\xi>\zeta$ so that $\mathscr{J}_\zeta\subsetneq \mathscr{J}_\xi$.   

\label{main theorem}
\end{theorem}

\section{A result of James}

Recall that a bounded subset $S$ of $X$ is relatively weakly compact if and only if $\overline{S}^{w^* \text{\ in\ }X^{**}}\subset X$, where $X$ is identified with its image under the canonical embedding into $X^{**}$.  From this we deduce that an operator $A:X\to Y$ fails to be weakly compact precisely when there exists $y^{**}\in \overline{AB_X}^{w^*}\setminus Y$.  The following result is due to James, and it is at the heart of our characterization.

\begin{proposition}\cite{James}  If $A:X\to Y$ is any operator between Banach spaces and if $y^{**}\in \overline{AB_X}^{w^*}\setminus Y$, then for any $0<\ee < \|y^{**}\|_{Y^{**}/Y}$, $B_X$ admits an infinite $(A^*B_{Y^*}, \ee)$-cs sequence.  
\label{james1}
\end{proposition}

Of course, it is clear that if $(x_i)\subset B_X$ is $(A^*B_{Y^*}, \ee)$-cs, no convex block of $(Ax_i)$ can be norm convergent, which means no subsequence of $(Ax_i)$ can be weakly convergent.  Thus $A:X\to Y$ fails to be weakly compact precisely when there exists $\ee>0$ so that $J( A^*B_{Y^*},\ee)$ is ill-founded, when happens if and only if there exists $\ee>0$ so that $j( A^*B_{Y^*}, \ee)=\infty$.  Thus we obtain the following portion of Theorem \ref{main theorem}.  

\begin{corollary} An operator $A:X\to Y$ is weakly compact if and only if $\mathcal{J}(A)<\infty$.  Therefore $\cup_{\xi\in \emph{\ord}}\mathscr{J}_\xi$ is the ideal of weakly compact operators.  

\end{corollary}

For the next result, we recall that an operator $A:X\to Y$ is called \emph{super weakly compact} if whenever $\uuu$ is an ultrafilter, the induced operator $A_\uuu:X_\uuu\to Y_\uuu$ between the ultrapowers is weakly compact.  

We note the following, which follows from standard ultrafilter techniques:  If $0<\delta<\ee$, if $A:X\to Y$ is an operator, and if $\uuu$ is any free ultrafilter on $\nn$, $$j(A^*B_{Y^*}, \ee)>\omega \Rightarrow j(A^*_\uuu B_{Y^*_\uuu}, \ee) =\infty \Rightarrow j(A^*B_{Y^*}, \delta)>\omega.$$ To see this, note that $j(A^*B_{Y^*}, \ee)>\omega$ implies the existence of $(x^n_i)_{i=1}^n\in B_X^{<\nn}$ so that for all $n\in \nn$ and $1\leqslant m<n$, the convex hulls of $(Ax_i^n)_{i=1}^m$ and $(Ax_i^n)_{i=m+1}^n$ have norm distance at least $\ee$ from each other.  For each $n\in \nn$ and $i>n$, let $x^n_i=0$.  Then the equivalence class $\chi_i\in B_{X_\uuu}$ containing $(x^n_i)_n$ is such that the convex hulls of $(A_\uuu \chi_i)_{i=1}^m$ and $(A_\uuu \chi_i)_{i=m+1}^\infty$ have norm distance at least $\ee$ from each other for all $m\in \nn$.  Conversely, if $(\chi_i)\in B_{X_\uuu}^\nn$ has the latter property, then we can find for all $n\in \nn$ some sequence $(x_i^n)_{i=1}^n\in B_X^{<\nn}$ having the former convex separation property with $\ee$ replaced by $\delta$ (recall that $0<\delta<\ee$ were fixed constants). This means that $j(A^*B_{Y^*}, \delta)\geqslant \omega$.  Since $J(A^*B_{Y^*}, \delta)$ includes the empty sequence, $j(A^*B_{Y^*}, \delta)$ cannot be a limit ordinal, so $j(A^*B_{Y^*}, \delta)>\omega$.  These observations immediately yield the next portion of Theorem \ref{main theorem}.

\begin{corollary} Let $A:X\to Y$ be an operator.  Then $\mathcal{J}(A)\leqslant \omega$ if and only if $A$ is super weakly compact.  
\label{characterize}

\end{corollary}

\begin{proposition} Let $X,Y$ be Banach spaces, $\xi\in \emph{\ord}$, and $A:X\to Y$ be an operator.  \begin{enumerate}[(i)]\item For any Banach space $W$ and any operator $B:W\to X$, $\mathcal{J}(AB)\leqslant \mathcal{J}(A)$.  \item For any Banach space and any operator $B:Y\to W$, $\mathcal{J}(BA)\leqslant \mathcal{J}(A)$. \item $\mathcal{J}_\xi(X,Y)$ is norm closed in $\mathfrak{L}(X,Y)$, the space of operators from $X$ to $Y$ endowed with operator norm. \end{enumerate}
\label{mostly ideal}
\end{proposition}

\begin{proof}[Proof of Proposition \ref{mostly ideal}] Both $(i)$ and $(ii)$ are trivial in the case that $B$ is the zero operator, so we assume $B$ is not.  It is clear that we may assume $\|B\|=1$, since $j((cAB)^*B_{Y^*}, c \ee)=j((AB)^*B_{Y^*}, \ee)$ (resp. $j((cBA)^*B_{W^*}, c \ee)=j((BA)^*B_{W^*}, \ee)$) for each $c>0$.  

$(i)$ Given a finite sequence $t=(w_i)_{i=1}^n\in B_W^{<\nn}$, we let $B(t)=(Bw_i)_{i=1}^n\in B_X^{<\nn}$.  We let $B(\varnothing)=\varnothing$ and given a subset $T$ of $B_W^{<\nn}$, we let $B(T)=\{B(t): t\in T\}$.  Then one easily verifies that $B(J((AB)^*B_{Y^*}, \ee))\subset J(A^*B_{Y^*}, \ee)$, whence it follows that $j((AB)^*B_{Y^*}, \ee)=o(B(J((AB)^*B_{Y^*}, \ee)))\leqslant j(A^*B_{Y^*}, \ee)$.  Taking the supremum over all $\ee>0$ gives $(i)$.

$(ii)$ Since $\|B\|=1$, $(BA)^*B_{W^*}\subset A^*B_{Y^*}$, from which $J((BA)^*B_{W^*}, \ee)\subset J(A^*B_{Y^*}, \ee)$ and $j((BA)^*B_{W^*}, \ee)\leqslant j(A^*B_{Y^*}, \ee)$ and $\mathcal{J}(AB)\leqslant \mathcal{J}(A)$ easily follow.

$(iii)$ We will show that if $A\in \mathfrak{L}(X,Y)\setminus \mathscr{J}_\xi(X,Y)$, then some open ball around $A$ lies in $\mathfrak{L}(X,Y)\setminus \mathscr{J}_\xi(X,Y)$.  If $A\in \mathfrak{L}(X,Y)\setminus \mathscr{J}_\xi(X,Y)$, then $j(A^*B_{Y^*}, \ee)>\xi$ for some $\ee>0$.  Fix $\delta_1, \delta_2>0$ such that $2\delta_1+\delta_2=\ee$.   We will show that if $\|A-B\|<\delta_1$, $j(B^*B_{Y^*}, \delta_2)>\xi$, and therefore $B\in \mathfrak{L}(X,Y)\setminus \mathscr{J}_\xi(X,Y)$.  To that end, it suffices to show that $J(A^*B_{Y^*}, \ee)\subset J(B^*B_{Y^*}, \delta_2)$ if $\|A-B\|<\delta_1$.  Fix $(x_i)_{i=1}^n\in J(A^*B_{Y^*}, \ee)$.  For $1\leqslant m<n$ and non-negative scalars $(a_i)_{i=1}^n$ so that $1=\sum_{i=1}^m a_i=\sum_{i=m+1}^n a_i$, \begin{align*} \|\sum_{i=1}^m a_iBx_i - \sum_{i=m+1}^n a_iBx_i\| & \geqslant \|\sum_{i=1}^m a_i A x_i - \sum_{i=m+1}^n a_iAx_i\| \\ & - \sum_{i=1}^m a_i\|A-B\| - \sum_{i=m+1}^n a_i \|A-B\| \\ & \geqslant \ee - 2\delta_1 =\delta_2.\end{align*}

\end{proof}

In order to complete the proof that $\mathscr{J}_{\omega^{\omega^\xi}}$ is an ideal, we need only to show that it is closed under finite sums.  For this, we need one more lemma, the proof of which will comprise the final section of this work.

\begin{lemma} Let $X, Y$ be Banach spaces and $A,B:X\to Y$ operators.  Then if $\mathcal{J}(A+B, \ee)>\omega^{\omega^\xi}$, then either $\mathcal{J}(A, \ee/3)>\omega^{\omega^\xi}$ or $\mathcal{J}(B, \ee/3)>\omega^{\omega^\xi}$. \label{sum lemma} \end{lemma}

\begin{corollary} For each $\xi\in \emph{\ord}$, $\mathscr{J}_{\omega^{\omega^\xi}}$ is a closed operator ideal.  

\end{corollary}

We recall that in \cite{BCFW}, the Bourgain $\ell_1$ index of an operator was defined.  Given $A:X\to Y$ and $K\geqslant 1$, we let $T_1(A,X,Y, K)$ consist of all of those finite sequences in $B_X$ (including the empty sequence) so that the image of each sequence under $A$ satisfies a $K$ lower $\ell_1$ estimate. By this, we mean sequences $(x_i)_{i=1}^n\in B_X^{<\nn}$ such that $\|\sum_{i=1}^n a_i Ax_i\|\geqslant K^{-1}\sum_{i=1}^n |a_i|$ for all scalars $(a_i)_{i=1}^n$.  It is obvious that $T_1(A,X,Y, K)\subset J(A^*B_{Y^*}, K^{-1})$, whence $\textbf{NP}_1(A,X,Y):=\sup_{K\geqslant 1} o(T_1(A,X,Y,K))\leqslant \mathcal{J}(A)$.   Moreover, it was shown in \cite{BCFW} that for any ordinal $\xi$, there exists a reflexive Banach space $W_\xi$ so that the identity $I_{W_\xi}$ on $W_\xi$ satisfies $\textbf{NP}_1(I_{W_\xi}, W_\xi, W_\xi)>\omega^\xi$.  Thus these examples yield that for any ordinal $\xi$, there exists a weakly compact operator which does not lie in $\mathscr{J}_\xi$.   Moreover, for $\xi$ countable, $W_\xi$ can be taken to be separable.   Thus the classes $\mathscr{J}_\xi$ exhaust the ideal of weakly compact operators, but each $\mathscr{J}_\xi$ is properly contained within the ideal of weakly compact operators.  This observation completes the proof of Theorem \ref{main theorem}. 

We note that for any body $K$, the tree $J(K, \ee)$ is closed.  That is, for any $n\in \nn$ and for any sequence $((x^j_i)_{i=1}^n)_j$ of members of $J(K, \ee)$, and $(x_i)_{i=1}^n\in B_X^{<\nn}$ such that $x^j_i\underset{j}{\to }x_i$ for each $1\leqslant i\leqslant n$, $(x_i)_{i=1}^n\in J(K, \ee)$.  To see this, simply note that if $1\leqslant m<n$ is fixed and if for each $j\in \nn$,  $x^*_j\in K$ is chosen so that $x^*_j(x_i-x_k)\geqslant \ee$ for each $1\leqslant i \leqslant m<k\leqslant n$, then any $w^*$ cluster point $x^*$ of $(x^*_j)$ satisfies $x^*(x_i-x_k)\geqslant \ee$ for $1\leqslant i\leqslant m<k\leqslant n$.  Since our definition of body included $w^*$ compactness, $x^*\in K$.  Therefore if $X$ is separable, Bourgain's version of the Kunen-Martin theorem \cite{B} implies that  $J(K,\ee)$ is either ill-founded or has countable order.  Thus if $X$ is separable and $A:X\to Y$ is an operator, $A$ is weakly compact if and only if $\mathcal{J}(A)<\omega_1$.

As observed in the previous paragraph, $\textbf{NP}_1(A, X, Y)\leqslant \mathcal{J}(A)$.  In general, however, there is no way to bound $\mathcal{J}(A)$ by $\textbf{NP}_1(A,X,Y)$.  A somewhat obvious example of this is James space, $J$, which fails to be reflexive, but also fails to contain a copy of $\ell_1$.  Therefore $\textbf{NP}_1(I_J, J, J)$ is countable, while $\mathcal{J}(I_J)=\infty$.  More interesting examples are furnished by subspaces of the James tree space, $JT$.  Let $(e_t)_{t\in \nn^{<\nn}}$ denote the canonical Hamel basis for $c_{00}(\nn^{<\nn})$, where $\nn^{<\nn}$ is the set of all finite sequences of natural numbers.  By a \emph{segment}, we shall mean a (possibly empty) subset of $\nn^{<\nn}$ of the form $\{t: s\preceq t \preceq u\}$.  Recall that $s\preceq t$ means that $s$ is an initial segment of $t$.  Then the James tree space $JT$ is the completion of $c_{00}(\nn^{<\nn})$ under the norm defined by $$\|\sum_{t\in \nn^{<\nn}} a_t e_t\| = \sup\Bigl\{ \Bigl(\sum_{i=1}^n \bigl|\sum_{t\in \mathfrak{s}_i} a_t\bigr|^2\Bigr)^{1/2}: (\mathfrak{s}_i)_{i=1}^n \text{\ are disjoint segments}\Bigr\}.$$  Then $JT$ does not contain a copy of $\ell_1$, so $\textbf{NP}_1(I_{JT}, JT, JT)$ is countable.  However, for any well-founded tree $T$ on $\nn$, the closed span $[e_t: t\in T]$ is reflexive (which can easily be seen by induction on the order of $T$).  However, it is easy to see that the identity $I_T$ on the subspace $[e_t:t\in T]$ of $JT$ has $\mathcal{J}(I_T)\geqslant o(T)$.  Thus the subspaces of $JT$ furnish reflexive examples of Banach spaces with no uniform countable bound on the $\mathcal{J}$ index of their identities, while there does exist a uniform, countable bound on the Bourgain $\ell_1$ indices of these spaces, namely the value $\textbf{NP}_1(I_{JT}, JT, JT)$.

More generally, for any set $S$, we may define the James tree space $JT(S)$, where the Hamel basis $(e_t)_{t\in \nn^{<\nn}}$ is replaced by the Hamel basis $(e_t)_{t\in S^{<\nn}}$.  If $S= [1, \xi]$, there exists a well-founded tree $\mathcal{MT}_\xi$ (for a specific example of such a tree, see \cite{C}) on $[1,\xi]$ with $o(\mathcal{MT}_\xi)=\xi+1$, and the tree $(e_t)_{t\in \mathcal{MT}_\xi}$ witnesses the fact that $j([e_t:t\in \mathcal{MT}_\xi], 1)>\xi$, where $\xi$ may be uncountable.  As before, we deduce $[e_t: t\in \mathcal{MT}_\xi]$ is reflexive (actually, $[e_t: t\in T]$ is reflexive whenever $T$ is well-founded, which may be shown by induction on the order as in the countable case).  However, the $\ell_1$ index of $JT(S)$ cannot exceed the $\ell_1$ index of $JT$.  To see this, note that if $\zeta$ is the $\ell_1$ index of $JT$, $\zeta$ is countable. If the $\ell_1$ index of $JT(S)$ exceeded $\zeta$, then there would be a separable subspace $X$ of $JT(S)$ having $\ell_1$ index exceeding $\zeta$.  But then there would exist a countable subset $S_0$ of $S$ so that $X\subset [e_t: t\in S_0^{<\nn}]$.  But the latter space is isometrically isomorphic to a subspace of $JT$.   Thus there exists an ordinal $\zeta$ such that for any ordinal $\xi$, there exists a reflexive Banach space having $\mathcal{J}$ index exceeding $\xi$, but having $\ell_1$ index not exceeding $\zeta$, and we deduce that the $\mathcal{J}$ index cannot be controlled by the $\ell_1$ index.  Similarly, since every block sequence in $JT$ dominates the $\ell_2$ basis, $JT$ cannot contain a copy of $c_0$, and we deduce that we cannot control the $\mathcal{J}$ index by the $c_0$ index.  We may similarly deduce results for uncountable indices by passing to $JT(S)$ and repeating the arguments for $\ell_1$.

\section{Descriptive set theoretic results}

We wish to recall the coding of the class of operators between separable Banach spaces, modeled on Bossard's coding of the class of separable Banach spaces \cite{Bos}.  Let $C(2^\nn)$ denote the space of all continuous functions on the Cantor set.  Recall that $\textbf{SB}$ denotes the space of all closed subspaces of $C(2^\nn)$, endowed with the Effros-Borel structure, and that this structure is standard.  That is, there exists a Polish topology on $\textbf{SB}$ such that the Borel $\sigma$-algebra generated by this topology is the Effros-Borel $\sigma$-algebra.  We fix such a topology on $\textbf{SB}$ to which we omit direct reference. Recall also \cite{KRN} that there exists a sequence $d_n:\textbf{SB}\to C(2^\nn)$ of Borel functions, called \emph{selectors}, such that for each $X\in \textbf{SB}$, $d_n(X)\in \textbf{SB}$ and $D_X:=\{d_n(X):n\in \nn\}$ is dense in $X$.  Recall also the definition of the space $\mathfrak{L}\subset \textbf{SB}\times \textbf{SB}\times C(2^\nn)^\nn$ defined in \cite{BF} by $(X,Y, \hat{A})\in \mathfrak{L}$ if and only if $\hat{A}(n)\in Y$ for all $n\in \nn$ and there exists $k\in \nn$ such that $\|\sum_{i=1}^p q_i\hat{A}(n_i)\|\leqslant k\|\sum_{i=1}^p q_i d_{n_i}(X)\|$ for all sequences $(n_i)_{i=1}^p\in \nn^{<\nn}$ and all rationals $(q_i)_{i=1}^p$.  Then $\mathfrak{L}$ codes the space of all operators between separable Banach spaces by taking $A:X\to Y$ to $(X, Y, (Ad_n(X)))$ for $X,Y\in \textbf{SB}$. By an abuse of notation, we identify operators with triples in this way.  Moreover, $\mathfrak{L}$ is a Borel subset of $\textbf{SB}\times \textbf{SB}\times C(2^\nn)^\nn$, and therefore it is also standard.  We arrive at the following.

\begin{proposition} The index $(X, Y, \hat{A})\mapsto \mathcal{J}(A)$ is a coanalytic rank on $\mathfrak{L}\cap \mathfrak{WC}$, where $\mathfrak{WC}$ denotes the ideal of weakly compact operators.

\end{proposition}

\begin{remark} It follows from this result that for any countable $\xi$, $\{(X, Y, \hat{A})\in \mathfrak{L}: \mathcal{J}(A)\leqslant \xi\}$ is a Borel subset of $\mathfrak{L}$, and for any analytic subset $\mathcal{A}$ of $\mathfrak{L}$, $\sup\{\mathcal{J}(A): A\in \mathcal{A}\}$ is countable.  Moreover, it follows from the proof, which involves a Borel reduction of the weakly compact operators to the well-founded trees on $\nn$, that $ \mathfrak{L}\cap \mathfrak{WC}$ is coanalytic in $\mathfrak{L}$.  These facts concerning coanalytic ranks can be found in \cite{D}.

\end{remark}

\begin{remark} It was shown in \cite{BF} that $ \mathfrak{L}\cap \mathfrak{WC}$ is coanalytic complete, and in particular non-Borel, in $\mathfrak{L}$.

\end{remark}

\begin{proof} Let \textbf{Tr} denote the trees on $\nn$, topologized with the relative topology inherited from $2^{\nn^{<\nn}}$. Let $\textbf{WF}\subset \textbf{Tr}$ denote the well-founded trees in $\textbf{Tr}$.   To show that $\mathcal{J}$ is a coanalytic rank, it suffices to show that the map $$(X, Y, \hat{A})\underset{f}{\mapsto} \{\varnothing\}\cup \{(k):k\in \nn\}\cup \bigl\{k\cat (n_i)_{i=1}^p: (d_{n_i}(X))_{i=1}^p\in J(A^*B_{Y^*}, k^{-1})\bigr\}$$ is Borel, $f^{-1}(\textbf{WF})=\mathfrak{L}\cap \mathfrak{WC}$, and that $o(f(X,Y, \hat{A}))= \mathcal{J}(A)+1$ \cite{D}.  The second and third facts follow from an inessential modification of the a similar argument from \cite{BCFW} concerning the indices $\textbf{NP}_p$.  To see that $f$ is Borel, as argued in \cite{BCFW}, it is sufficient to fix $t=k\cat (n_i)_{i=1}^p$ and prove that $\{(X,Y, \hat{A}): t\in f(X,Y,\hat{A})\}$ is Borel.  But $t=k\cat (n_i)_{i=1}^p\in f(X,Y, \hat{A})$ if and only if $\|d_{n_i}(X)\|\leqslant 1$ for each $1\leqslant i\leqslant p$ and $\|\sum_{i=1}^m a_i \hat{A}(n_i) - \sum_{i=m+1}^p a_i \hat{A}(n_i)\|\geqslant 1/k$ for each $1\leqslant m<p$ and all non-negative, rational scalars $(a_i)_{i=1}^p$ such that $\sum_{i=1}^m a_i=\sum_{i=m+1}^pa_i =1$.  Since this is a collection of countably many Borel conditions, we deduce that $f$ is Borel.

\end{proof}

\section{Proof of Lemma \ref{sum lemma}}

\subsection{The Hessenberg sum, results on simple colorings, tree multiplication}

Recall that any ordinal $\xi$ can be uniquely written $$\xi=\omega^{\alpha_1} n_1+\ldots \omega^{\alpha_k} n_k$$ for ordinals $\alpha_1>\ldots >\alpha_k$ and natural numbers $n_i$ (where $k=0$ corresponds to $\xi=0$) \cite{M}.  This is called the \emph{Cantor normal form} of $\xi$.  If $\xi, \zeta$ are two ordinals, by adding zero terms into the Cantor normal forms of $\xi$ and $\zeta$, we can express $$\xi=\omega^{\alpha_1} m_1+\ldots +\omega^{\alpha_k} m_k,$$ $$\zeta=\omega^{\alpha_1} n_1+\ldots + \omega^{\alpha_k} n_k,$$ where the same ordinals $\alpha_i$ appear in the expressions.  We then define the \emph{Hessenberg} or \emph{natural} sum by $$\xi\oplus \zeta= \omega^{\alpha_1}(m_1+n_1)+\ldots +\omega^{\alpha_k}(m_k+n_k).$$  

Because it is rather inconvenient to include the empty sequence in our proofs below, we will be concerned with subsets $T$ of $S^{<\nn}\setminus \{\varnothing\}$ such that $T\cup \{\varnothing\}$ is a tree.  Such sets are called $B$-\emph{trees}.  The notions of derived $B$-trees and orders can be relativized to $B$-trees.  If $U,V$ are subsets of $S_1^{<\nn}$, $S_2^{<\nn}$, respectively, a function $\theta:U\to V$ is called \emph{monotone} if for each $s,t\in U$ with $s\prec t$, $\theta(s)\prec \theta(t)$.  Given a monotone map $\theta:U\to V$, we say a function $e:MAX(U)\to MAX(V)$ is an \emph{extension} of $\theta$ if for each $s\in MAX(U)$, $\theta(s)\preceq e(s)$.  We say a pair $(\theta, e):U\times MAX(U)\to V\times MAX(V)$ is an \emph{extended monotone map} if $\theta$ is monotone and $e$ is an extension of $e$. To avoid cumbersome notation, we will often say an extended monotone map $(\theta, e):U\times MAX(U)\to V\times MAX(V)$ is from $U$ to $V$, rather than $U\times MAX(U)$ to $V\times MAX(V)$, or write $(\theta, e):U\to V$ in place of $(\theta, e):U\times MAX(U)\to V\times MAX(V)$.   Note that if $V$ is a non-empty subset of a well-founded $B$-tree and if $\theta:U\to V$ is any monotone map, then there exists an extension $e$ of $\theta$.  

The following method for ``multiplying'' non-empty, well-founded $B$-trees is inspired by the ``replacement trees'' defined in \cite{JO}.  Recall the convention that for sets $S_1, S_2$, $(S_1\times S_2)^{<\nn}$ is identified with $\{(s,t)\in S_1^{<\nn}\times S_2^{<\nn}: |s|=|t|\}$.  That is, we identify sequences of pairs with pairs of sequences via $(a_i, b_i)_{i=1}^n\leftrightarrow ((a_i)_{i=1}^n, (b_i)_{i=1}^n)$.  We identify $\varnothing$ with $(\varnothing, \varnothing)$.  Given a member $x$ of $S$ and $n\in \nn$, we let $x^{(n)}$ denote the constant sequence in $S$ which has length $n$ and begins with $x$.  If $T_0,T_1$ are non-empty, well-founded $B$-trees, we let $[T_0, T_1]$ consist of all concatenations $$t=(s_1, x_1^{(|s_1|)})\cat (s_2, x_2^{(|s_2|)})\cat \ldots \cat (s_k, x_k^{(|s_k|)})$$ such that \begin{enumerate}[(i)]\item $k\in \nn$, \item $s_i\in T_0$ for each $1\leqslant i\leqslant k$,\item for each $1\leqslant i<k$, $s_i\in MAX(T_0)$, \item $(x_i)_{i=1}^k\in T_1$. \end{enumerate} If $s_k\in MAX(T_0)$ as well, we will say $t$ is \emph{regular} in $[T_0, T_1]$.  The intuition behind this construction is to build a ``tree of trees,'' where we think of beginning with the tree $T_0$ and replacing each of its members with a tree isomorphic to $T_1$.  To understand this, note that if $t=(s_1, x_1^{(|s_1|)})\cat (s_2, x_2^{(|s_2|)})\cat \ldots \cat (s_k, x_k^{(|s_k|)})\in [T_0, T_1]$ is regular, and if $(x_i)_{i=1}^{k+1}\in T_1$ for some $x_{k+1}$, then the function taking $s\in T_0$ to $t\cat (s, x_{k+1}^{(|s|)})$ is an order isomorphism of $T_0$ with a subset of $[T_0, T_1]$, and the minimal regular extensions of $t$ in $[T_0, T_1]$ are precisely the images of members of $MAX(T_0)$ under this map. Therefore the portion of the tree $[T_0, T_1]$ consisting of proper extensions of a fixed regular member $t$ which are initial segments of the minimal regular extensions of $t$ can be canonically identified with $T_0$. Similarly, the initial segments of the set of minimal regular members of $[T_0, T_1]$ can be canonically identified with $T_0$ (this is analogous to the previous sentence with $t$ replaced by $\varnothing$).   This observation gives an intuition for why the derived trees of $[T_0, T_1]$ satisfy $[T_0, T_1]^{o(T_0) \xi} = [T_0, T^\xi_1]$, which can be easily shown by induction on $\xi$.  This fact yields that $o([T_0, T_1])=o(T_0)o(T_1)$.  The interested reader is invited to compare this process with the convolution of regular families and its effects on Cantor-Bendixson index, discussed, for example, in \cite{LT}.  

Of particular interest to us will be the $B$-trees $\ttt_\xi$, defined in \cite{C}.  We let $\ttt_1=\{(1)\}$, $\ttt_{\xi+1}= \{(\xi+1), (\xi+1)\cat t: t\in \ttt_\xi\}$, and $\ttt_\xi=\cup_{\eta<\xi} \ttt_{\eta+1}$ when $\xi$ is a limit.  Note that when $\xi$ is a limit ordinal, $\ttt_\xi$ is a totally incomparable union.   We will also be interested in the $B$-trees $[\ttt_\xi, \ttt_k]$ when $k\in \nn$.  Note that in this case, an arbitrary member $t$ of $[\ttt_\xi, \ttt_k]$ can be written uniquely as $$(t_1, k^{(|t_1|)})\cat (t_2, (k-1)^{(|t_2|)})\cat \ldots \cat (t_i, (k-i+1)^{(|t_i|)}).$$  In this case, we say that $t$ is in the $i^{th}$ \emph{level} of $[\ttt_\xi, \ttt_k]$.  In this case, the first level is naturally order isomorphic to $\ttt_\xi$ via the map $t\leftrightarrow (t, k^{(|t|)})$.  Moreover, for $i<k$, if $t$ is maximal in the $i^{th}$ level, then the proper extensions of $t$ which lie in the $(i+1)^{st}$ level, a set which we will call the \emph{unit under} (or \emph{beneath}) $t$, form a set naturally ordre isomorphic to $\ttt_\xi$ via the identification $s\leftrightarrow t\cat (s, (k-i)^{(|s|)})$.  A \emph{unit} of $[\ttt_\xi, \ttt_k]$ will refer either to the first level of $[\ttt_\xi, \ttt_k]$ or the unit under $t$ for some $t$.  Note that two comparable members of the same level of $[\ttt_\xi, \ttt_k]$ must lie in the same unit, the first level of $[\ttt_\xi, \ttt_k]$ is a single unit, and for any $1<i\leqslant k$, the $i^{th}$ level of $[\ttt_\xi, \ttt_k]$ is the totally incomparable union of the units beneath the maximal members of the $(i-1)^{st}$ level.  

More generally, if $t$ is a maximal member of the $i^{th}$ level of $[\ttt_\xi, \ttt_k]$ where $k-i=j$, then the proper extensions of $t$ in $[\ttt_\xi, \ttt_k]$ form a set naturally order isomorphic to $[\ttt_\xi, \ttt_j]$ via the identification $$t\cat (t_i, (k-i)^{(|t_i|)})\cat \ldots \cat (t_m, (k-m)^{(|t_m|)})\leftrightarrow (t_i, j^{(|t_i|)}) \cat (t_{i+1}, (j-1)^{(|t_{i+1}|)}) \cat \ldots \cat (t_m, (j-m+1)^{(|t_m|)}),$$  valid for each $1\leqslant m\leqslant j$.  Similarly, the first $j$ levels of $[\ttt_\xi, \ttt_k]$ are naturally order isomorphic to $[\ttt_\xi, \ttt_j]$ (this is similar to the previous fact with $t=\varnothing$).

  The following result is easily shown by induction on $\zeta$ for $\xi$ held fixed.  This result can also be deduced from the well-known and easy to see result that for $B$-trees $T_0$ and $T_1$, there exists a monotone function $\theta:T_0\to T_1$ if and only if $o(T_0)\leqslant o(T_1)$ (see \cite{Ke}).  

\begin{proposition} For $0<\xi, \zeta\in \emph{\ord}$, there exist monotone maps $\theta_1:[\ttt_\xi, \ttt_\zeta]\to \ttt_{\xi\zeta}$ and $\theta_2:\ttt_{\xi\zeta}\to [\ttt_\xi, \ttt_\zeta]$.  

\end{proposition}

The following facts were shown in \cite{C}.

\begin{proposition} 
\begin{enumerate}[(i)]\item For any $0<\xi\in \emph{\ord}$ and any finite partition $\mathcal{P}$ of $MAX(\ttt_\xi)$, there exists an extended monotone map $(\theta, e)$ of $\ttt_\xi$ into $\ttt_\xi$ and $P\in \mathcal{P}$ so that $e(MAX(\ttt_\xi))\subset P$.  \item For any $\xi\in \emph{\ord}$, any finite set $S$, and any function $$f:\{(s,t)\in \ttt_{\omega^\xi}\times MAX(\ttt_{\omega^\xi}): s\preceq t\}\to S,$$ there exists $x\in S$ and an extended monotone map $(\theta, e)$ of $\ttt_{\omega^\xi}$ into itself so that for all $s\preceq t\in MAX(\ttt_{\omega^\xi})$, $f(\theta(s), e(t))=x$.  \item For any $\xi\in \emph{\ord}$, any finite set $S$, and any function $f:\ttt_{\omega^\xi}\to S$, there exists a monotone map $\theta:\ttt_{\omega^\xi}\to \ttt_{\omega^\xi}$ such that $f\circ \theta$ is constant. \item For any $\xi\in \emph{\ord}$, any $n\in \nn$, and $$f:\{(s,t)\in \ttt_{\omega^\xi(2n-1)}\times MAX(\ttt_{\omega^\xi(2n-1)}): s\preceq t\}\to 2,$$ there exists $\ee\in 2$ and an extended monotone map $(\theta,e):\ttt_{\omega^\xi n}\to \ttt_{\omega^\xi (2n-1)}$ so that for all $s\preceq t\in MAX(\ttt_{\omega^\xi n})$, $\ee= f(\theta(s), e(t))$. \end{enumerate}

\label{recall}
\end{proposition}

We deduce the following, equivalent to Proposition \ref{recall}$(iv)$.

\begin{proposition} For any $\xi\in \emph{\ord}$, $n\in \nn$, and $$f:\{(s,t)\in [\ttt_{\omega^\xi}, \ttt_{2n-1}]\times MAX([\ttt_{\omega^\xi}, \ttt_{2n-1}]): s\preceq t\} \to 2,$$ there exists an extended monotone $(\theta,e)$ of $[\ttt_{\omega^\xi}, \ttt_n]$ into $[\ttt_{\omega^\xi}, \ttt_{2n-1}]$ and $\ee\in 2$ such that for all $s\preceq t\in MAX([\ttt_{\omega^\xi}, \ttt_{2n-1}])$, $f(\theta(s), e(t))=\ee$.    

\label{replacement prop}
\end{proposition}

\begin{proof} Fix extended monotone maps $(\theta_1, e_1):\ttt_{\omega^\xi (2n-1)}\to [\ttt_{\omega^\xi}, \ttt_{2n-1}]$ and $(\theta_2,e_2):[\ttt_{\omega^\xi}, \ttt_n]\to \ttt_{\omega^\xi n}$.  Let $f':\{(s,t)\in \ttt_{\omega^\xi (2n-1)}\times MAX(\ttt_{\omega^\xi (2n-1)}): s\preceq t\}\to 2$ be given by $f'(s,t)=f(\theta_1(s), e_1(t))$.  By Proposition \ref{replacement prop}, we can fix an extended monotone $(\theta',e'):\ttt_{\omega^\xi n}\to \ttt_{\omega^\xi (2n-1)}$ and $\ee\in 2$ so that $f'(\theta'(s), e'(t))=\ee$ for all $s\preceq t\in MAX(\ttt_{\omega^\xi (2n-1)})$.  Let $(\theta,e)=( \theta_1\circ\theta'\circ\theta_2, e_1\circ e'\circ e_2)$.

\end{proof}

\subsection{The main lemma}

Suppose $T$ is a $B$-tree.  Define $\Lambda(T)=\{(s,t)\in T\times T: s\prec t\}$.  The goal of this subsection is to discuss such colorings when $S$ is finite and how to find ``large subtrees'' $T_0$ of $T$ such that $f|_{\Lambda(T_0)}$ is constant. To that end, we have the following result. 

\begin{lemma} For any ordinal $\xi$, any $B$-tree $T$ with $o(T)\geqslant \omega^\xi$, and any function $f:\Lambda(T)\to 2$, there exist ordinals $\xi_0$ and $\xi_1$ such that $\xi_0\oplus \xi_1=\xi$, $B$ trees $T_0$, $T_1$ with $o(T_0)=\omega^{\xi_0}$, $o(T_1)=\omega^{\xi_1}$, and for $\ee=0,1$, there exist monotone maps $\theta_\ee:T_\ee\to T$ so that for all $(s,t)\in \Lambda(T_\ee)$, $f(\theta_\ee(s), \theta_\ee(t))=\ee$.  

\label{combinatorial1}
\end{lemma} 

The fact that the only pairs $(\beta, \gamma)$ with $\beta\oplus \gamma=\omega^\xi$ are $(0, \omega^\xi)$ and $(\omega^\xi, 0)$, together with the usual method of reducing colorings to colorings with strictly fewer colors yields the following.  

\begin{corollary} If $o(T)\geqslant \omega^{\omega^\xi}$, if $S$ is any finite set, and if $f:\Lambda(T)\to S$ is any function, then there exists a $B$-tree $T_0$ with $o(T_0)=\omega^{\omega^\xi}$, a monotone map $\theta:T_0\to T$, and $x\in S$ so that $f(\theta(s), \theta(t))= x$ for all $(s,t)\in \Lambda(T)$.  
\label{main corollary}
\end{corollary}

\begin{proof} The statement preceding the corollary yields the result if $|S|\leqslant 2$.  Assume we have the result for all finite sets $S'$ with $|S'|\leqslant 2^k$. Fix a set $S$ with $|S|\leqslant 2^{k+1}$ and $f:T\to S$.  Let $S_0$, $S_1$ be a partition of $S$ into two subsets each with cardinality not exceeding $2^k$.  Define $g(s,t)=0$ if $f(s,t)\in S_0$ and $g(s,t)=1$ if $f(s,t)\in S_1$.  Then choose a $B$ tree $T_0$ with $o(T_0)=\omega^{\omega^\xi}$, $\ee\in 2$, and a monotone map $\theta:T_0\to T$ so that $g(\theta(s), \theta(t))=\ee$ for all $(s,t)\in \Lambda(T_0)$.  Next, define $f':\Lambda(T_0)\to S_\ee$ by $f'(s,t)=f(\theta(s), \theta(t))$.  Apply the inductive hypothesis to obtain $T_1$ with $o(T_1)=\omega^{\omega^\xi}$, an $x\in S_\ee$, and a monotone map $\theta':\Lambda(T_1)\to S_\ee$ so that $f'(\theta'(s), \theta'(t))=x$ for all $(s,t)\in \Lambda(T_1)$.  Then $x$, $T_1$, and $\theta\circ \theta':T_1\to T$ satisfy the conclusions.

\end{proof}

\begin{remark} Recall that if $T_0$ is any $B$-tree with $o(T_0)\geqslant \xi>0$, there exists a monotone map $\theta:\ttt_\xi\to T_0$.  Thus if $\ee\in 2$, $f:\Lambda(T)\to 2$, and $\theta_0:T_0\to T$ are such that $f(\theta_0(s), \theta_0(t))=\ee$ for all $(s,t)\in \Lambda(T_0)$, then $f(\theta_0\circ \theta(s), \theta_0\circ \theta(t))=\ee$.  Thus in the conclusion of Lemma \ref{combinatorial1}, we may assume $T_0=\ttt_{\omega^{\xi_0}}$, and $T_1=\ttt_{\omega^{\xi_1}}$.

Moreover, if $f:T\to 2$ is such that $o(T)\geqslant \omega^\xi$, as in the hypothesis of Lemma \ref{combinatorial1}, we may first take a monotone map $\theta:\ttt_{\omega^\xi}\to T$ and replace $f:\Lambda(T)\to 2$ with $f':\Lambda(\ttt_{\omega^\xi})\to 2$ given by $f'(s,t)=f(\theta(s), \theta(t))$.  Then if we can find $T_0, T_1$, and monotone maps $\theta_0:T_0\to \ttt_{\omega^\xi}$ and $\theta_1:T_1\to\ttt_{\omega^\xi}$ as in $(ii)$ of the conclusion, then $\theta\circ \theta_0:T_0\to T$ and $\theta\circ \theta_1:T_1\to T$ satisfy the conclusion.  Thus the lemma holds if and only if it holds when $T=\ttt_{\omega^\xi}$.  

We will implicitly use these facts throughout.  It will be very convenient, however, to allow the trees $T$, $T_0$, and $T_1$ to be other trees besides $\ttt_\xi$ for some $\xi$, so we do not state the lemma in this way.  

\end{remark}

\begin{remark} The base case of Lemma \ref{combinatorial1} is equivalent to the finite Ramsey theorem from \cite{Ramsey}: For any $n\in \nn$, there exists $N=N(n)\in \nn$ so that for any $N\leqslant M\in \nn$ and any function $f:\{(i,j): 1\leqslant i<j\leqslant M\}\to 2$, there exist $1\leqslant p_1<\ldots <p_n\leqslant M$ and $\ee\in 2$ so that $f(p_i, p_j)=\ee$ for all $1\leqslant i\leqslant j$.   We simply note that $\ttt_M$ is order isomorphic to $\{1, \ldots, M\}$, and so there is a natural bijection between $\Lambda(\ttt_M)$ and $\{(i,j):1\leqslant i<j\leqslant M\}$.  If $f:\ttt_\omega\to 2$ is any function, then for each $n\in \nn$, we find a monotone map $\theta_n:\ttt_n\to \ttt_{N(n)}$ and $\ee_n$ so that $f(\theta_n(t_i), \theta_n(t_j))=\ee_n$ for each $1\leqslant i<j\leqslant n$, where $\ttt_n=\{t_1, \ldots, t_n\}$, $t_1\prec \ldots \prec t_n$.  We then choose $n_1<n_2<\ldots$ and $\ee\in 2$ so that $\ee_{n_i}=\ee$ for all $i\in \nn$, and let $T_\ee=\cup_i \ttt_{n_i}$.  Then the monotone map $\theta_\ee$ from $T_\ee$ to $T$ is given by $\theta_\ee|_{\ttt_{n_i}}=\theta_{n_i}$. Then $o(T_\ee)=\omega^1$, and $\xi_\ee=1$ we set.  The monotone map $\theta_{1-\ee}:\ttt_1\to \ttt_\omega$ given by mapping the unique member of $\ttt_1$ to any member of $\ttt_\omega$ vacuously statisfies the condition required of it, since $\Lambda(\ttt_1)$ is empty.

Moreover, the ill-founded analogue of Lemma \ref{combinatorial1} is just the infinite Ramsey theorem.  The ill-founded analogue would be that if $f:\Lambda(T)\to 2$ is any function, where $T$ is ill-founded, then there exists $\ee\in 2$, an ill-founded tree $T_0$, and a monotone map $\theta:T_0\to T$ so that $f(\theta(s), \theta(t))=\ee$ for all $(s,t)\in \Lambda(T_0)$.  This is precisely equivalent to the statement: For any function $f:\{(i,j)\in \nn\times \nn: i<j\}\to 2$, there exists $\ee\in 2$ and natural numbers $m_1<m_2<\ldots$ so that $f(m_i, m_j)=\ee$ for all $1\leqslant i<j$.  This is because a tree is ill-founded if and only if there exists $(t_i)_{i=1}^\infty\subset T$ so that $t_1\prec t_2\prec \ldots$.  Then $f'(i,j)=f(t_i, t_j)$ defines a $2$ coloring of $\{(i,j): i<j\}$, and we may define a monotone map from $T_0=\{(1, \ldots, n): n\in \nn\}$ by $\theta((1, \ldots, n))= t_{m_n}$, where $m_1<m_2<\ldots$ is such that $f'(m_i, m_j)=\ee$ for all $i<j$.

\end{remark}

\begin{rem}
\upshape
We will prove Lemma \ref{combinatorial1} by induction on $\xi$.  We have already argued the base case. We have already noted that if $\zeta=0$, the existence of a monotone $\theta:\ttt_{\omega^\xi}\to T$ so that $f(\theta(s), \theta(t))$ is constant on $\Lambda(\ttt_{\omega^\xi})=\Lambda(\ttt_1)=\varnothing$ is trivial.  For $\zeta>0$, the existence of $\ee$, $T_\ee$ with $o(T_\ee)=\omega^{\xi_\ee}$, and a monotone map $\theta_\ee:T_\ee\to T$ satisfying the conclusions of Lemma \ref{combinatorial1} is equivalent to the existence of a subset $A$ of $[0, \omega^{\xi_\ee})$ with $\sup A=\omega^{\xi_\ee}$ and for each $\zeta\in A$ the existence of $T_{\ee, \zeta}$ with $o(T_{\ee, \zeta})=\zeta$ and a monotone map $\theta_{\ee, \zeta}:T_{\ee, \zeta}\to T$ so that $f(\theta_{\ee, \zeta}(s), \theta_{\ee,\zeta}(t))=\ee$ for all $(s,t)\in \Lambda(T_{\ee, \zeta})$.  The tree $T_\ee$ can then be taken to be a totally incomparable union of the $B$-trees $T_{\ee, \zeta}$ (or, formally, $B$-trees $\tilde{T}_{\ee, \zeta}$ which are order isomorphic to $T_{\ee, \zeta}$ but made to be totally incomparable), the map $\theta$ will be equal to $\theta_{\ee, \zeta}$ when restricted to $T_{\ee, \zeta}$, and $o(T_0)=\sup_{\zeta\in A}o(T_{\ee,\zeta})=\omega^{\xi_\ee}$.  This fact will be prevalent, so we isolate it to avoid repetition during our proofs.   
\label{useful remark}
\end{rem}

The proof of the limit ordinal case is quite easy, with a fact from \cite{C} concerning Hessenberg sums.  Assuming the result for each $\zeta<\xi$, $\xi$ a limit ordinal, and fixing a $B$-tree $T$ with $o(T)\geqslant \omega^\xi$ and a function $f:\Lambda(T)\to 2$, we note that $o(T)> \omega^{\zeta+1}$ for each $\zeta<\xi$.  By the inductive hypothesis, for each $\zeta<\xi$, there exist ordinals $\zeta_0$ and $\zeta_1$ with $\zeta_0\oplus \zeta_1=\zeta+1$, $B$-trees $T_{0, \zeta}$ and $T_{1, \zeta}$ so that $o(T_{0, \zeta})=\omega^{\zeta_0}$ and $o(T_{1, \zeta})=\omega^{\zeta_1}$, and for $\ee=0,1$, monotone maps $\theta_{\ee, \zeta}:T_{\ee, \zeta}\to T$ so that $f(\theta_{\ee, \zeta}(s), \theta_{\ee, \zeta}(t))=\ee$ for all $(s,t)\in \Lambda(T_{\ee, \zeta})$.  Then by \cite[Proposition 2.5]{C}, there exist a set $A\subset [0, \omega^\xi)$ and $\xi_0$, $\xi_1$ so that $\xi_0\oplus \xi_1=\xi$ and for some $\ee\in 2$, $\sup_{\zeta\in A}\zeta_\ee = \xi_\ee$ and $\min_{\zeta\in A} \zeta_{1-\ee}\geqslant \xi_{1-\ee}$.  Assume for convenience that $\ee=0$.  Then by Remark \ref{useful remark}, $(T_{0, \zeta})_{\zeta\in A}$ and $(\theta_{0, \zeta})_{\zeta\in A}$ guarantee the existence of the desired $T_0$ with $o(T_0)=\omega^{\xi_0}$ and $\theta_0:T_0\to T$.  Moreover, we may take $T_1$ to be $\ttt_{\omega^{\xi_1}}$.  Fix any $\zeta\in A$ and a monotone $\theta':\ttt_{\omega^{\xi_1}}\to T_{1, \zeta}$, which we may do since $o(T_{1, \zeta})\geqslant \omega^{\xi_1}$.  Then $\theta_1:\ttt_{\omega^{\xi_1}}\to T$ can be given by $\theta_{1, \zeta}\circ \theta'$.  This completes the limit ordinal case.

For the successor case, we will do some preliminary work.

\begin{proposition} Fix $\xi\in \emph{\ord}$, $k,n\in \nn$, and a finite set $S$. Fix $f:\Lambda([\ttt_{\omega^\xi}, \ttt_k])\to 2$ and a function $g:[\ttt_{\omega^\xi}, \ttt_k]\to S$ such that if $s$ and $t$ lie in the same unit of $[\ttt_{\omega^\xi}, \ttt_k]$, $g(s)=g(t)$.  \begin{enumerate}[(i)]\item If $k=2^{2n-2}$, there exists a monotone map $\theta_1:[\ttt_{\omega^\xi}, \ttt_n]\to [\ttt_{\omega^\xi}, \ttt_k]$ and $\ee\in 2$ such that if $(s,t)\in \Lambda([\ttt_{\omega^\xi}, \ttt_n])$, and $s,t$ lie on different levels, $f(\theta_1(s), \theta_1(t))=\ee$.  \item If $k=n|S|$, there exists a monotone map $\theta_2:[\ttt_{\omega^\xi}, \ttt_n]\to [\ttt_{\omega^\xi}, \ttt_k]$ and $x\in S$ so that $g\circ \theta_2\equiv x$, and $\theta_2$ restricted to a unit of $[\ttt_{\omega^\xi}, \ttt_n]$ is an order isomorphism with a unit of $[\ttt_{\omega^\xi}, \ttt_k]$, and for some $1\leqslant l_1<\ldots<l_n\leqslant k$, $\theta_2$ maps the $i^{th}$ level of $[\ttt_{\omega^\xi}, \ttt_n]$ to the $l_i^{th}$ level of $[\ttt_{\omega^\xi}, \ttt_k]$.      \end{enumerate}

\label{long one}
\end{proposition}

\begin{proof}$(i)$ We first claim that if $k=2^{n-1}$, there exists a monotone map $\theta':[\ttt_{\omega^\xi}, \ttt_n]\to [\ttt_{\omega^\xi}, \ttt_k]$ and $(\ee_j)_{j=1}^n\in 2^n$ so that if $(s,t)\in \Lambda([\ttt_{\omega^\xi}, \ttt_n])$, $s$ on level $i$, $t$ on level $j$, and $i<j$, $f(\theta'(s), \theta'(t))=\ee_j$.  Assume this claim.  Then if $f$ is as in $(i)$ with $k=2^{2n-2}$, let $m=2n-1$ so that $k=2^{m-1}$.  Then by the claim, there exists a monotone map $\theta':[\ttt_{\omega^\xi}, \ttt_m]\to 2$ and $(\ee_j)_{j=1}^m\in 2^m$ satisfying the conclusions of the claim.  Then we can choose $1\leqslant l_1<\ldots <l_n\leqslant m$ and $\ee\in 2$ so that $\ee=\ee_{l_j}$ for each $1\leqslant j\leqslant n$.  Let $U_\varnothing$ be the first unit of $[\ttt_{\omega^\xi}, \ttt_n]$ and let $\theta''$ be defined on $U_\varnothing$ by being an order isomorphism with any unit on the $l_1$ level of $[\ttt_{\omega^\xi}, \ttt_m]$.  Next, suppose $\theta''$ has been defined on the first $i$ levels of $[\ttt_{\omega^\xi}, \ttt_n]$ for some $i<n$ so that $\theta''$ maps the $j^{th}$ level of $[\ttt_{\omega^\xi}, \ttt_n]$ into the $l_j^{th}$ level of $[\ttt_{\omega^\xi}, \ttt_m]$. Fix $t$ maximal in the $i^{th}$ level of $[\ttt_{\omega^\xi}, \ttt_n]$ and let $U_t$ be the unit beneath $t$. Let $u$ be an extension of $\theta''(t)$ which is a maximal member of the $(l_{i+1}-1)^{st}$ level of $[\ttt_{\omega^\xi}, \ttt_k]$.   Let $\theta''|_{U_t}$ be an order isomorphism between $U_t$ and the unit $U_u$ beneath $u$.  This completes the recursive construction of $\theta'':[\ttt_{\omega^\xi}, \ttt_n]\to [\ttt_{\omega^\xi}, \ttt_m]$.  Then $\theta= \theta'\circ \theta''$ clearly satisfies the conclusion of $(i)$.  

We return to the proof of the claim at the beginning of the previous paragraph.  We prove the result by induction on $n$.  If $n=1$, the conclusion is vacuous when $\theta$ is the identity, since there is only one level. Assume the result holds for a given $n\in \nn$ and let $f:\Lambda([\ttt_{\omega^\xi}, \ttt_{2^n}])\to 2$ be as in the statement of the claim.  Let $T=[\ttt_{\omega^\xi}, \ttt_{2^n}]^{\omega^\xi}$, which is the first $2^n-1$ levels of $[\ttt_{\omega^\xi}, \ttt_{2^n}]$ and is naturally order isomorphic to $[\ttt_{\omega^\xi}, \ttt_{2^n-1}]$.  For $t\in MAX(T)$ and $U_t$ the unit beneath $t$, define $f_t:U_t\to 2^{|t|}$ by letting $f_t(u)=(f(t|_i, u))_{i=1}^{|t|}$. Using Proposition \ref{recall} and noting that $U_t$ can be identified with $\ttt_{\omega^\xi}$, we deduce the existence of $(\ee^t_i)_{i=1}^{|t|}\in 2^{|t|}$ and a monotone map $\theta_t:U_t\to U_t$ so that $(f(t|_i, \theta_t(u)))_{i=1}^{|t|}=(\ee_i^t)_{i=1}^{|t|}$ for all $u\in U_t$.  

Define $f':T\times MAX(T)\to 2$ by letting $f'(s,t)=\ee^t_{|s|}$ and note that $f'(s,t)=f(s, \theta_t(u))$ for any proper extension $u$ of $t$.    Using Proposition \ref{replacement prop} and recalling that $T$ is order isomorphic to $[\ttt_{\omega^\xi}, \ttt_{2^n-1}]$, we deduce the existence of a monotone map $\theta:[\ttt_{\omega^\xi}, \ttt_{2^{n-1}}]\to T$ and an extension map $e:MAX([\ttt_{\omega^\xi}, \ttt_{2^{n-1}}])\to MAX(T)$ of $\theta$ and $\ee_{n+1}\in 2$ so that $\ee_{n+1}=f'(\theta(s), e(t))$ for any $s\preceq t\in MAX([\ttt_{\omega^\xi}, \ttt_{2^{n-1}}])$.  Let $f'':\Lambda([\ttt_{\omega^\xi}, \ttt_{2^{n-1}}])\to 2$ be given by $f''(s,t)=f(\theta(s), \theta(t))$.  Applying the inductive hypothesis, we obtain $(\ee_i)_{i=1}^n\in 2^n$ and monotone map $\theta'':[\ttt_{\omega^\xi}, \ttt_n]\to [\ttt_{\omega^\xi}, \ttt_{2^{n-1}}]$ so that for $s\prec t$, $s$ in the $i^{th}$ level, $t$ on the $j^{th}$ level, and $i<j$, $f''(\theta''(s), \theta''(t))=\ee_j$.  We last define $\theta':[\ttt_{\omega^\xi}, \ttt_{n+1}]\to [\ttt_{\omega^\xi}, \ttt_{2^n}]$.  Let $\iota:[\ttt_{\omega^\xi}, \ttt_n]\to [\ttt_{\omega^\xi}, \ttt_{n+1}]$ be the natural order isomorphism between $[\ttt_{\omega^\xi}, \ttt_n]$ and the first $n$ levels of $[\ttt_{\omega^\xi}, \ttt_{n+1}]$.  Let $\theta'$ be defined on the first $n$ levels of $[\ttt_{\omega^\xi}, \ttt_{n+1}]$ by letting $\theta'= \theta\circ\theta''\circ \iota^{-1}$ on those levels. It is straightforward to verify that for $s\prec t$, $s$ on the $i^{th}$ level, $t$ on the $j^{th}$ level, $1\leqslant i<j\leqslant n$, $f(\theta'(s), \theta'(t))=\ee_j$.   Fix $t$ maximal in the $n^{th}$ level of $[\ttt_{\omega^\xi}, \ttt_{n+1}]$ and let $U_t$ be the unit beneath $t$. Let $U_{e(\theta''\circ \iota^{-1}(t))}$ be the unit beneath $e(\theta''\circ \iota^{-1}(t))$ and let $\iota_t$ be the natural order isomorphism between $U_t$ and $U_{e(\theta''\circ \iota^{-1}(t))}$.  Let $\theta|_{U_t}=\theta_{e(\theta''\circ \iota^{-1}(t))}\circ \iota_t$.   Fix $s\prec u$ with $s$ on the $i^{th}$ level for some $i\leqslant n$ and $u$ on the $(n+1)^{st}$ level.  Then there exists a unique $t$ which is a maximal member of the $n^{th}$ level of $[\ttt_{\omega^\xi}, \ttt_{n+1}]$ and such that $s\preceq t \prec u$.  Then by the properties of $(\theta, e)$ and the first sentence at the beginning of the paragraph (with $s$ replaced by $\theta''\circ \iota^{-1}(s)$, $t$ replaced by $e(\theta''\circ \iota^{-1}(t))$, and $u$ replaced by $\theta_{e(\theta''\circ \iota^{-1}(t))}\circ \iota_t(u)$),  $$f(\theta'(s), \theta'(u)) = f(\theta(\theta''\circ \iota^{-1}(s)), \theta_{e(\theta''\circ \iota^{-1}(t))}\circ \iota_t(u))= f'(\theta(\theta''\circ \iota^{-1}(s)), e(\theta''\circ \iota^{-1}(t)))=\ee_{n+1}.$$

$(ii)$ We first claim that for any $k$ and any function $g$ as in the statement of Proposition \ref{long one}, there exists a monotone map $\theta:[\ttt_{\omega^\xi}, \ttt_k]\to [\ttt_{\omega^\xi}, \ttt_k]$ taking the $i^{th}$ level to the $i^{th}$ level and so that the restriction of $\theta$ to a unit is an order isomorphism with a unit of $[\ttt_{\omega^\xi}, \ttt_k]$, and so that there exists $(x_i)_{i=1}^k\subset S^k$ so that if $s$ is on the $i^{th}$ level of $[\ttt_{\omega^\xi}, \ttt_k]$, $g(\theta(s))=x_i$.  The result in the case that $k=n|S|$ then follows by the pigeonhole principle, which guarantees the existence of $1\leqslant l_1<\ldots <l_n\leqslant n|S|$ and $x\in S$ so that $x=x_{l_i}$ for all $1\leqslant i\leqslant n$, and using the method from part $(i)$ for defining a monotone map from $[\ttt_{\omega^\xi}, \ttt_n]$ to $[\ttt_{\omega^\xi}, \ttt_k]$ mapping the $i^{th}$ level to the $l_i^{th}$ level and then composing this map with the $\theta$ from the claim.  

We prove the claim by induction on $k$.  For $k=1$, the result holds trivially, since the hypothesis guarantees that $g$ is constant on the single unit of $[\ttt_{\omega^\xi}, \ttt_1]$.  Assume the result holds for a given $k$ and assume $g:[\ttt_{\omega^\xi}, \ttt_{k+1}]\to S$ is as in the statement.  Fix $t$ maximal in the first level.  Note that the proper extensions of $t$ form a set naturally order isomorphic to $[\ttt_{\omega^\xi}, \ttt_k]$. Let $E_t$ denote the proper extensions of $t$ in $[\ttt_{\omega^\xi}, \ttt_{k+1}]$ and define $g_t:E_t\to S$ by letting $g_t(s)=g(s)$.  Identyfing $E_t$ with $[\ttt_{\omega^\xi}, \ttt_k]$ (as we may by previous remarks) and using the inductive hypothesis, we obtain a monotone $\theta_t:E_t\to E_t$ and a finite sequence $(x^t_i)_{i=2}^{k+1}\subset S^k$ so that if $t\prec s$ and $s$ is on the $i^{th}$ level of $[\ttt_{\omega^\xi}, \ttt_{k+1}]$, $x^t_i= g(\theta_t(s))$.  For each $x\in S^k$, let $C_x$ consist of those $t$ which are maximal in the first level of $[\ttt_{\omega^\xi}, \ttt_{k+1}]$ such that $x=(x^t_i)_{i=2}^{k+1}$.  Identifying the first level with $\ttt_{\omega^\xi}$, by Proposition \ref{recall}, there exists an extended monotone map $(\theta_0, e_0)$ of the first level of $[\ttt_{\omega^\xi}, \ttt_{k+1}]$ into itself and $(x_i)_{i=2}^{k+1}=x\in S^k$ so that the range of $e_0$ is contained in $C_x$.  We then define $\theta$ on $[\ttt_{\omega^\xi},\ttt_{k+1}]$ by making it equal to $\theta_0$ on the first level.  Fix $t$ maximal in the first level of $[\ttt_{\omega^\xi}, \ttt_{k+1}]$.  Letting again $E_t$ denote the proper extensions of $t$ in $[\ttt_{\omega^\xi}, \ttt_{k+1}]$, and $E_{e_0(t)}$ denote the proper extensions of $e_0(t)$, we note that $E_t$ is naturally order isomorphic to $E_{e_0(t)}$.  Let $p_t:E_t\to E_{e_0(t)}$ denote the natural order isomorphism, and let $\theta|_{E_t}= \theta_{e_0(t)}\circ p_t$.  This $\theta$ is clearly seen to satisfy the claim with the sequence $(\ee_i)_{i=1}^{k+1}$, where $\ee_1$ is the common value of $g$ on the first level of $[\ttt_{\omega^\xi}, \ttt_{k+1}]$, which consists of a single unit.

\end{proof}

\begin{proof}[Proof of Lemma \ref{combinatorial1}, successor case] Assume Lemma \ref{combinatorial1} holds for a given $\xi\geqslant 1$ and fix a $B$-tree $T$ with $o(T)\geqslant \omega^{\xi+1}$ and $f:T\to 2$ any function.  Let $S=\{(\gamma_0, \gamma_1): \gamma_0\oplus \gamma_1=\xi\}$, and note that $S$ is finite.  We first claim that there exist natural numbers $n_1<n_2<\ldots$, $\ee\in 2$, a pair $(\xi_0, \xi_1)\in S$, and monotone maps $\theta_i':[\ttt_{\omega^\xi}, \ttt_{n_i}]\to T$ so that \begin{enumerate}[(i)]\item for each $i\in \nn$ and $(s,t)\in \Lambda([\ttt_{\omega^\xi}, \ttt_{n_i}])$ with $s,t$ on different levels, $f(\theta_i'(s),\theta_i'(t))=\ee$, \item for each $i\in \nn$ and each unit $U$ of $[\ttt_{\omega^\xi}, \ttt_{n_i}]$, there exist monotone maps $\phi_0:\ttt_{\omega^{\xi_0}}\to U$ and $\phi_1:\ttt_{\omega^{\xi_1}}\to U$ so that for $j=0,1$ and $(s,t)\in [\ttt_{\omega^{\xi_j}}, \ttt_{n_i}]$, $f(\theta_i'\circ\phi_j(s),\theta_i'\circ \phi_j(t))=j$.  \end{enumerate}

We first show how this finishes the proof, and then we show this claim.   Suppose for convenience that the $\ee$ in the claim is equal to $0$. Then fix any $i\in \nn$ and let $U$ be the first unit of $[\ttt_{\omega^\xi}, \ttt_{n_i}]$.  Let $\phi_1:\ttt_{\omega^{\xi_1}}\to U$ be as in (ii) of the claim.  Then  $T_1=\ttt_{\omega^{\xi_1}}$ and $\theta_1:T_1\to T$ defined by $\theta_1= \theta_i'\circ\phi_1$ satisfy the $j=1$ conclusion of Lemma \ref{combinatorial1}. We will construct for each $i\in \nn$ a monotone map $\varphi_i:[\ttt_{\omega^{\xi_0}}, \ttt_{n_i}]\to [\ttt_{\omega^\xi}, \ttt_{n_i}]$ so that $f(\theta'_i\circ \varphi_i(s), \theta'_i\circ \varphi_i(t))=0$ for all $(s,t)\in \Lambda([\ttt_{\omega^{\xi_0}}, \ttt_{n_i}])$. Then the maps $\theta_i'\circ \varphi_i:[\ttt_{\omega^{\xi_0}}, \ttt_{n_i}]\to T$, along with an appeal to Remark \ref{useful remark}, guarantee the existence of the $T_0$ and $\theta_0$ required for the conclusion of Lemma \ref{combinatorial1}.  Here we use the fact that $\omega^{\xi_0+1}=\sup_{i\in \nn} \omega^{\xi_0} n_i= \sup_{i\in \nn} o([\ttt_{\omega^{\xi_0}}, \ttt_{n_i}])$.

We define $\varphi_i$ by induction on the levels of $[\ttt_{\omega^{\xi_0}}, \ttt_{n_i}]$.  Let $U_\varnothing$ be the first unit of $[\ttt_{\omega^\xi}, \ttt_{n_i}]$.  Then by (ii) above, there exists some monotone $\phi_\varnothing:\ttt_{\omega^{\xi_0}}\to U_\varnothing$ such that for $(s,t)\in \Lambda(\ttt_{\omega^{\xi_0}})$, $0=f(\theta_i'\circ \phi_\varnothing(s), \theta_i'\circ \phi_\varnothing(t))$.  We let $V_\varnothing$ denote the first level of $[\ttt_{\omega^{\xi_0}}, \ttt_{n_i}]$ and $\iota_\varnothing:V_\varnothing\to \ttt_{\omega^{\xi_0}}$ be the natural order isomorphism and define $\varphi_i$ on $V_\varnothing$ to be $ \phi_0 \circ  \iota_\varnothing$.  Next, assume $\varphi_i$ has been defined on the first $k$ levels of $[\ttt_{\omega^{\xi_0}}, \ttt_{n_i}]$ for some $k<n_i$ and that $\varphi_i$ takes the $j^{th}$ level of $[\ttt_{\omega^{\xi_0}}, \ttt_{n_i}]$ to the $j^{th}$ level of $[\ttt_{\omega^\xi}, \ttt_{n_i}]$ for each $1\leqslant j\leqslant k$.  Fix $t$ maximal in the $k^{th}$ level of $[\ttt_{\omega^{\xi_0}}, \ttt_{n_i}]$, and let $u$ be an extension of $\varphi_i(t)$ which is a maximal member of the $k^{th}$ level of $[\ttt_{\omega^\xi}, \ttt_{n_i}]$.  Let $V_t$ be the unit beneath $t$ and let $U_u$ be the unit beneath $u$.  Let $\iota_t:V_t\to \ttt_{\omega^{\xi_0}}$ be the natural order isomorphism.  Fix some monotone $\phi_t:\ttt_{\omega^{\xi_0}}\to V_u$ as in (ii) and let $\varphi_i$ be equal to $\phi_t\circ\iota_t$ on $V_t$. This completes the recursive construction, and it is clear that $\varphi_i$ has the announced property. If $s\prec u$ for $s,u$ lying in the same unit $U$ of $[\ttt_{\omega^{\xi_0}}, \ttt_{n_i}]$, then $f(\varphi_i(s), \varphi_i(u))= f(\varphi'_i\circ \varphi_t\circ \iota_t(s), \theta'_i\circ \varphi_t\circ \iota_t(u))=0$ for some $t$ ($t$ may either be regular if $s,u$ are not in the first level, or $t=\varnothing$ if $s,t$ lie in the first level).  If $s\prec u$ for $s,u$ in different levels of $[\ttt_{\omega^\xi}, \ttt_{n_i}]$, then $f(\varphi_i(s), \varphi_i(u))=f(\theta'_i\circ \phi_{t_1}\circ \iota_{t_1}(s), \theta'_i\circ \phi_{t_2}\circ \iota_{t_2}(s))=0$ for some $t_1, t_2$.  This follows from (i) together with the fact that $\phi_{t_1}\circ \iota_{t_1}(s)$ and $\phi_{t_2}\circ \iota_{t_2}(u)$ lie on different levels when $s$ and $u$ do.  

We last complete the proof of the claim.  Fix $n\in \nn$ and fix a monotone map $p_n:[\ttt_{\omega^\xi}, \ttt_{2^{2n|S|-2}}]\to T$.  Let $f':\Lambda([\ttt_{\omega^\xi}, \ttt_{2^{2n|S|-2}}])\to 2$ be given by $f'(s,t)=f(p_n(s), p_n(t))$.  Then by Proposition \ref{long one}$(i)$, there exist $\ee_n\in 2$ and a monotone $q_n:[\ttt_{\omega^\xi}, \ttt_{n|S|}]\to [\ttt_{\omega^\xi}, \ttt_{2^{2n|S|-2}}]$ so that for $(s,t)\in \Lambda([\ttt_{\omega^\xi}, \ttt_{n|S|}])$ on different levels, $f'(q_n(s), q_n(t))=f(p_n\circ q_n(s), p_n\circ q_n(t))=\ee_n$.  Define $g:[\ttt_{\omega^\xi}, \ttt_{n|S|}]\to S$ as follows: For a unit $U$ of $[\ttt_{\omega^\xi}, \ttt_{n|S|}]$, let $\iota_U:\ttt_{\omega^\xi}\to U$ be the natural order isomorphism.  Let $f''(s,t)=f(p_n\circ q_n\circ \iota_U(s), p_n\circ q_n \iota_U(t))$ for each $(s,t)\in \Lambda(\ttt_{\omega^\xi})$.  By the inductive hypothesis, there exists a pair $(\xi_0(n,U), \xi_1(n,U))\in S$ and for $j=0,1$, monotone maps $\phi_{j,U}:\ttt_{\omega^{\xi_j(n,U)}}\to U$ so that $f''(\phi_{j,U}(s), \phi_{j,U}(t))=j$.  Let $g(s)=(\xi_0(n,U), \xi_1(n,U))$ for every $s\in U$.  This defines a $g$ as in Proposition \ref{long one}$(ii)$, whence we deduce the existence of a monotone $r_n:[\ttt_{\omega^\xi}, \ttt_n]\to [\ttt_{\omega^\xi}, \ttt_{n|S|}]$ and a pair $(\xi_0(n), \xi_1(n))$ such that $r_n$ restricted to any unit is an order isomorphism with a unit of $[\ttt_{\omega^\xi}, \ttt_{n|S|}]$ and so that $g\circ r_n\equiv (\xi_0(n), \xi_1(n))$.  Moreover, this $r_n$ maps distinct levels to distinct levels.  Choosing $n_1<n_2<\ldots$, $\ee\in 2$, and $(\xi_0, \xi_1)\in S$ so that $\ee=\ee_{n_i}$ and $(\xi_0, \xi_1)=(\xi_0(n_i), \xi_1(n_i))$ for all $i\in \nn$ is easily seen to satisfy the conclusion with $\theta_i'= p_{n_i}\circ q_{n_i}\circ r_{n_i}$.  

\end{proof}

\begin{remark} 
We remark that the proof of Lemma \ref{combinatorial1} makes it easy to construct examples showing that the result is sharp.  That is: For any $\xi\in \ord$ and for any pair $(\zeta_0, \zeta_1)$ such that $\zeta_0\oplus \zeta_1=\xi$, there exists a $B$-tree with $o(T)=\omega^\xi$ and a function $f:\Lambda(T)\to 2$ so that if $\xi_0, \xi_1\in \ord$ are ordinals with $\xi_0\oplus\xi_1=\xi$, $T_0$, $T_1$ are $B$-trees, and if $\theta_0:T_0\to T$, $\theta_1:T_1\to T$ are monotone maps such that for $\ee\in 2$, and $(s,t)\in \Lambda(T_\ee)$, $f(\theta_ee(s), \theta_\ee(t))=\ee$, then $o(T_0)\leqslant \omega^{\xi_0}$ and $o(T_1)\leqslant \omega^{\xi_1}$. We only sketch the proof, since we do not use this fact in the sequel. We do, however, remark that for any $T$ satisfying the assertion above and any monotone $\theta:\ttt_{\omega^\xi}\to T$, $f'(s,t)=f(\theta(s), \theta(t))$ defines a function $f':\Lambda(\ttt_{\omega^\xi})\to 2$ also satisfying the the claim above.  Therefore this claim is true if and only if it is true when $T$ is replaced by $\ttt_{\omega^\xi}$.

Of course, the $\xi=0$ case is trivial.  For the $\xi=1$ case, $(\zeta_0, \zeta_1)=(0,1)$ or $(0,1)$.  Let $f:\Lambda(\ttt_\omega)\to 2$ be constantly $0$ if $\zeta_0=1$ and $\zeta_1=0$, and constantly $1$ if $\zeta_0=0$ and $\zeta_1=1$.  

Next, assume the result holds for an ordinal $\xi$ and all pairs $(\zeta_0', \zeta_1')$ with $\zeta_0'\oplus \zeta_1'=\xi$.  Let $\zeta_0\oplus \zeta_1= \xi+1$.  Then one of $\zeta_0$ and $\zeta_1$ must be a successor, assume for convenience that $\zeta_0=\zeta_0'+1$. Fix a function $g:\ttt_{\omega^\xi}\to 2$ to satisfy the conclusion of this claim for the pair $(\zeta_0', \zeta_1)$. Note that $\cup_k [\ttt_{\omega^\xi}, \ttt_k]$ is a totally incomparable union. Define $f:\Lambda(\cup_k [\ttt_{\omega^\xi}, \ttt_k])\to 2$ by letting $f(s,t)=0$ when $s$ and $t$ lie on different levels of $[\ttt_{\omega^\xi}, \ttt_k]$ (note that, of course, if $s$ and $t$ are comparable, they must both lie in $[\ttt_{\omega^\xi},\ttt_k]$ for the same $k$), and by letting $f(s,t)=g(\iota(s), \iota(t))$, when $s,t$ lie in the same unit of $[\ttt_{\omega^\xi}, \ttt_k]$, where $\iota$ is the natural order isomorphism from the unit containing $s$ and $t$ to $\ttt_{\omega^\xi}$.  It is easy to check that the conclusion is satisfied by this construction.   

Assume the result holds for all $\eta<\xi$, $\xi$ a limit ordinal.  Express $\xi$ in its Cantor normal form as $\xi=\omega^{\alpha_1}n_1+\ldots + \omega^{\alpha_k} n_k$.  Fix $\xi_0=\omega^{\alpha_1}p_1+\ldots + \omega^{\alpha_k} p_k$ and $\xi_1=\omega^{\alpha_1}q_1+\ldots + \omega^{\alpha_k} q_k$ such that $\xi_0\oplus \xi_1=\xi$.  Since $n_i=p_i+q_i$ for each $1\leqslant i\leqslant k$ and $n_k>0$, either $p_k>0$ or $q_k>0$.  Assume for convenience that $p_k>0$.  Let $\zeta_0= \omega^{\alpha_1}p_1+\ldots +\omega^{\alpha_k} (p_k-1)$, and note that $\xi_0=\zeta_0+\omega^{\alpha_k}$.   For each $\eta<\omega^{\alpha_k}$, $(\zeta_0+\eta)\oplus \xi_1<\xi$.  Thus there exists a $B$-tree $T_\eta$ with $o(T_\eta)=\omega^{(\zeta_0+\eta)\oplus \xi_1}$ and a function $f_\eta:T_\eta\to 2$ satisfying the conclusions with respect to the pair $(\zeta_0+\eta, \xi_1)$. By replacing each $T_\eta$ with an order isomorphic tree, we may assume the union $T=\cup_{\eta<\omega^{\alpha_k}}T_\eta$ is a totally incomparable union. We then define $f:T\to 2$ by letting $f|_{T_\eta}= f_\eta$.  It is easy to check in this case that the conclusion is satisfied, since we are dealing with a totally incomparable union.

\end{remark}

We state a few more simple propositions which we need to complete the proof of Lemma \ref{sum lemma}. For a well-founded $B$-tree, we let $$\Lambda^e(T)=\{(s,t,v)\in T\times T\times MAX(T): s\prec t\preceq v\}.$$  For each $0<\xi$, let $\eee_\xi= \{(s,t)\in \ttt_\xi\times MAX(\ttt_\xi): s\preceq t\}$. 

 \begin{proposition} Let $S$ be a finite set. \begin{enumerate}[(i)]\item For any function $f:\eee_\xi\to S$, there exists an extended monotone map $(\theta, e):\ttt_\xi\to \ttt_\xi$ and a function $g:\ttt_\xi\to S$ so that for all $(s,t)\in \eee_\xi$, $g(s)=f(\theta(s), e(t))$.   \item For any $f:\Lambda^e(\ttt_\xi)\to S$, there exists a function $g:\Lambda(\ttt_\xi)\to S$ and an extended monotone map $(\theta_1, e_1):\ttt_\xi\to \ttt_\xi$ so that $g(s,t)=f(\theta_1(s), \theta_1(t), e_1(v))$ for all $(s,t,v)\in \Lambda^e(\ttt_\xi)$.  \end{enumerate}\label{reduction} \end{proposition}

The content of this proposition is that we can remove from $f$ the dependence on the second argument for $(i)$, and remove from $f$ the dependence on the third argument in $(ii)$.    

\begin{proof}$(i)$ Recall that $\ttt_{\xi+1}=\{(\xi+1)\}\cup \{(\xi+1)\cat s: s\in \ttt_\xi\}$, and note that $\ttt_\xi \ni s\leftrightarrow (\xi+1)\cat s\in \ttt_{\xi+1}$ is an order isomorphism between $\ttt_\xi$ and the proper extensions of $(\xi+1)$ in $\ttt_{\xi+1}$ when $\xi>0$.  In particular, $s$ is maximal in $\ttt_\xi$ if and only if $(\xi+1)\cat s$ is maximal in $\ttt_{\xi+1}$.  We will use these facts implicitly throughout the proof.

By induction.  For any finite $\xi$ the result is trivial, since each member of $\ttt_\xi$ has a unique maximal extension.  We can simply take $\theta$ and $e$ to be the identities and $g(s)=f(s,t)$, where $t$ is the unique maximal extension of $t$.   

Suppose $\xi$ is a limit ordinal and the result holds for every $\eta<\xi$.  Then since $\eee_{\eta+1}\subset \eee_\xi$, by considering $f|_{\eee_{\eta+1}}$, we may use the inductive hypothesis to obtain $g_\eta:\ttt_{\eta+1}\to S$ and an extended monotone map $(\theta_\eta, e_\eta):\ttt_{\eta+1}\to \ttt_{\eta+1}$ so that $g_\eta(s)= f(\theta_\eta(s), e_\eta(t))$ for all $(s,t)\in \eee_{\eta+1}$. Define $(\theta, e)$ and $g$ by letting $\theta|_{\ttt_{\eta+1}}= \theta_\eta$, $e|_{MAX(\ttt_{\eta+1})}=e_\eta$, and $g|_{\ttt_{\eta+1}}=g_\eta$.  

Suppose the result holds for a given $\xi>0$ and let $f:\eee_{\xi+1}\to S$ be any function.  For each $x\in S$, let $C_x=\{t\in MAX(\ttt_{\xi+1}): x= f((\xi+1), t)\}$.  This is a finite partition of $MAX(\ttt_{\xi+1})$, and Proposition \ref{recall} yields that there exists an extended monotone map $(\theta', e'):\ttt_{\xi+1}\to \ttt_{\xi+1}$ and $x\in S$ such that $e'(MAX(\ttt_{\xi+1}))\subset C_x$.  Let $g((\xi+1))=x$.  Next, let $f''(s,t)=f((\xi+1)\cat \theta'(s),(\xi+1)\cat e'(t))$ for $(s,t)\in \eee_\xi$.  By the inductive hypothesis, there exists $(\theta'', e''):\ttt_\xi\to \ttt_\xi$ and $g'':\ttt_\xi\to S$ such that for any $(s,t)\in \eee_\xi$, $g''(s)=f''(\theta''(s), e''(t))$. Let $\theta((\xi+1))=(\xi+1)$.   For $s\in \ttt_\xi$, let $\theta((\xi+1)\cat s)= \theta'((\xi+1)\cat \theta''(s))$ and let $g((\xi+1)\cat s)=g''(s)$.  For $s\in MAX(\ttt_\xi)$, let $e((\xi+1)\cat s)= e'((\xi+1)\cat e''(s))$. It is straightforward to verify that the conclusions are satisfied with these definitions.

$(ii)$ Again, the result is trivial for any finite $\xi$, since there are unique maximal extensions.  We let $\theta_1$ and $e_1$ be the identities and $g(s,t)=f(s,t,v)$, where $v$ is the unique maximal extension of $t$.

Assume the result holds for all $\eta<\xi$, $\xi$ a limit ordinal.  Then $\Lambda^e(\ttt_\xi)$ is simply the disjoint union $\cup_{\eta<\xi} \Lambda^e(\ttt_{\eta+1})$.  Apply the inductive hypothesis to obtain for each $\eta<\xi$ an extended monotone $(\theta_{1, \eta}, e_{1, \eta}):\ttt_{\eta+1}\to \ttt_{\eta+1}$ and $g_\eta:\Lambda(\ttt_{\eta+1})\to S$ such that for each $(s,t,v)\in \Lambda^e(\ttt_{\eta+1})$, $g_\eta(s,t)= f(\theta_{1, \eta}(s), \theta_{1, \eta}(t), e_{1, \eta}(v))$.  Let $g|_{\Lambda(\ttt_{\eta+1})}=g_\eta$, $\theta_1|_{\ttt_{\eta+1}}=\theta_{1, \eta}$, $e_1|_{MAX(\ttt_{\eta+1})}= e_{1, \eta}$.   

Assume the result holds for a given $\xi$ and let $f:\Lambda^e(\ttt_{\xi+1})\to S$ be a function. Define $f':\Lambda^e(\ttt_\xi)\to S$ by letting $f'(s,t,v)=f((\xi+1)\cat s, (\xi+1)\cat t, (\xi+1)\cat v)$.  By the inductive hypothesis, there exists $g':\Lambda(\ttt_\xi)\to S$ and an extended monotone map $(\theta', e')$ of $\ttt_\xi$ into $\ttt_\xi$ so that $f'(\theta'(s), \theta'(t), e'(v))= g'(s,t)$ for all $(s,t,v)\in \Lambda^e(\ttt_\xi)$.  Define $f'':\eee_\xi\to S$ by $$f''(s,t)= f((\xi+1), (\xi+1)\cat \theta'(s), (\xi+1)\cat e'(t)).$$   By part $(i)$, there exists an extended monotone $(\theta'', e'')$ of $\ttt_\xi$ into itself and $g'':\ttt_\xi\to S$ so that $g''(s)=f''(\theta''(s), e''(t))$ for all $(s,t)\in \eee_\xi$.  Let $\theta_1((\xi+1))= (\xi+1)$, $\theta_1((\xi+1)\cat s)= (\xi+1)\cat \theta'\circ\theta''(s)$, and $e_1((\xi+1)\cat t)= (\xi+1)\cat e'\circ e''(t)$.  Let $g((\xi+1), (\xi+1)\cat s)=g''(s)$ for $s\in \ttt_\xi$, and for $s\prec t$, $s,t\in \ttt_\xi$, let $g(s,t)=g'(\theta''(s), \theta''(t))$.  Again, verification of the conclusions is straightforward.  

\end{proof}

\begin{corollary} For any ordinal $\xi$, any finite set $S$, and any function $f:\Lambda^e(\ttt_{\omega^{\omega^\xi}})\to S$, there exists $x\in S$ and an extended monotone map $(\theta, e)$ of $\ttt_{\omega^{\omega^\xi}}$ into itself so that for all $(s,t,v)\in \Lambda^e(\ttt_{\omega^{\omega^\xi}})$, $f(\theta(s), \theta(t), e(v))=x$. 
\label{cut to the chase}
\end{corollary}

\begin{proof} By Proposition \ref{reduction}, there exists an extended monotone map $(\theta_1, e_1)$ of $\ttt_{\omega^{\omega^\xi}}$ into itself and a function $g:\Lambda(\ttt_{\omega^{\omega^\xi}})\to S$ such that for each $(s,t,v)\in \Lambda^e(\ttt_{\omega^{\omega^\xi}})$, $g(s,t)=f(\theta_1(s), \theta_1(t), e_1(v))$.   By Corollary \ref{main corollary}, there exists a monotone map $\theta_2$ of $\ttt_{\omega^{\omega^\xi}}$ into itself and $x\in S$ so that $g(\theta_2(s), \theta_2(t))=x$ for all $(s,t)\in \Lambda(\ttt_{\omega^{\omega^\xi}})$.  Let $e_2$ be any extension of $\theta_2$ and let $(\theta, e)=(\theta_1\circ \theta_2, e_1\circ e_2)$.

\end{proof}

\begin{lemma} If $\zeta$ is any limit ordinal, there exist monotone maps $\theta, \phi:\ttt_\zeta\to \ttt_\zeta$ so that \begin{enumerate}[(i)]\item for each $s\in \ttt_\zeta$, $\theta(s)\prec \phi(s)$, \item for each $s,t\in \ttt_\zeta$ with $s\prec t$, $\phi(s)\prec \theta(t)$.  \end{enumerate}
\label{shuffling}
\end{lemma}

These maps have the effect of ``shuffling'' the members of $\ttt_\zeta$ into $\ttt_\zeta$.  By this, we mean that for any $s\in \ttt_\zeta$, $$\theta(s|_1)\prec \phi(s|_1)\prec \theta(s|_2)\prec \phi(s|_2)\prec \ldots \prec \theta(s)\prec \phi(s).$$  The scheme of the proof is as follows: Since $[\ttt_2,\ttt_\zeta]$ and $\ttt_\zeta$ have the same order, and since $[\ttt_2, \ttt_\zeta]$ is essentially $\ttt_\zeta$ where all the nodes have been replaced by a pair of nodes, we can define $p(s)$ and $q(s)$ to be the sequences terminating at the first and second nodes, respectively, which took the place of the original node $s$.  

\begin{proof} 

For a limit ordinal $\zeta$, $2\zeta=\zeta$.   This means that $o([\ttt_2, \ttt_\zeta])=2 \zeta=\zeta$, and there exists a monotone map $\theta_1:[\ttt_2, \ttt_\zeta]\to \ttt_\zeta$.  Fix $s=(\xi_1, \ldots, \xi_n)\in \ttt_\zeta$ and define $p,q:\ttt_\zeta\to [\ttt_2, \ttt_\zeta]$ by \begin{align*} p(s) & = (2, \xi_1)\cat (1, \xi_1)\cat (2, \xi_2) \cat (1, \xi_2)\cat \ldots \cat (2, \xi_n) \\ & = ((2,1), \xi_1^{(2)})\cat ((2,1), \xi_2^{(2)}) \cat \ldots \cat (2, \xi_n)\end{align*} and \begin{align*} q(s)& = (2, \xi_1)\cat (1, \xi_1)\cat (2, \xi_2) \cat (1, \xi_2)\cat \ldots \cat (2, \xi_n)\cat (1, \xi_n) \\ & = ((2,1), \xi_1^{(2)})\cat ((2,1), \xi_2^{(2)}) \cat \ldots \cat ((2,1), \xi_n^{(2)}). \end{align*}  Then for all $s\in \ttt_\zeta$, $$p(s|_1)\prec q(s|_1)\prec \ldots \prec p(s)\prec q(s).$$  Let $\theta(s)=\theta_1(p(s))$ and $\phi(s)=\theta_1(q(s))$.

\end{proof}

We are now ready to prove Lemma \ref{sum lemma}.  

\begin{corollary} Let $X$ be a Banach space and let $K, L\subset X^*$ be bodies.  For $\ee>0$ and any ordinal $\xi$,  if $j(K+L, \ee)>\omega^{\omega^\xi}$, then $\max\{j(K, \ee/3), j(L, \ee/3)\}>\omega^{\omega^\xi}$.  In particular, if $A,B:X\to Y$ are operators such that $\mathcal{J}(A+B)>\omega^{\omega^\xi}$, then either $\mathcal{J}(A)>\omega^{\omega^\xi}$ or $\mathcal{J}(B)>\omega^{\omega^\xi}$.  

\end{corollary}

\begin{proof}

Choose $R>0$ so that $K, L\subset RB_{X^*}$.  Let $\mathcal{P}$ be a finite partition of $[-R, R]$ into sets of diameter not exceeding $\delta=\ee/24$.  If $j(K+L, \ee)>\omega^{\omega^\xi}$, there exists $(x_t)_{t\in \ttt_{\omega^{\omega^\xi}}}\subset B_X$ so that $(x_{t|_i})_{i=1}^{|t|}\in J(K+L, \ee)$ for all $t\in \ttt_{\omega^{\omega^\xi}}$ \cite{C}.  By the definition of $J(K+L, \ee)$, this means that for each $v\in MAX(\ttt_{\omega^{\omega^\xi}})$ and each $t\prec v$, there exists a functional $x^*_{t,v}\in K+L$ so that \begin{equation} x^*_{t,v}(x_s- x_u)\geqslant \ee \text{\ \ \ }\forall s\prec t \prec u\preceq v, v\in MAX(\ttt_{\omega^{\omega^\xi}}). \end{equation}  Write $x^*_{t,v}=y^*_{t,v}+z^*_{t,v}$ with $y^*_{t,v}\in K$ and $z^*_{t,v}\in L$.   Define a function $f:\Lambda^e(\ttt_{\omega^{\omega^\xi}})\to \mathcal{P}\times \mathcal{P}$ by letting $f_1(s,t,v)= (U,V)$, where $(U,V)$ is the member of $\mathcal{P}\times \mathcal{P}$ such that $(y^*_{t,v}(x_s), z^*_{t,v}(x_s))\in U\times V$.  By Corollary \ref{cut to the chase}, there exists an extended monotone map $(\theta_1, e_1)$ of $\ttt_{\omega^{\omega^\xi}}$ into itself and a pair $(U,V)\in \mathcal{P}\times \mathcal{P}$ so that for all $(s,t,v)\in \Lambda^e(\ttt_{\omega^{\omega^\xi}})$, $y^*_{\theta_1(t), e_1(v)}(x_{\theta_1(s)})\in U$ and $z^*_{\theta_1(t), e_1(v)}(x_{\theta_1(s)})\in V$.  Note that we may replace $x_s$ with $x_{\theta_1(s)}$, $y^*_{t,v}$ with $y^*_{\theta_1(t), e_1(v)}$, and $z^*_{t,v}$ with $z^*_{\theta_1(t), e_1(v)}$ and assume that \begin{equation} (y^*_{t,v}+z^*_{t,v})(x_s-x_u)\geqslant \ee \text{\ \ }\forall s\prec t \prec u\preceq v, v\in MAX(\ttt_{\omega^{\omega^\xi}}),\end{equation} \begin{equation} y^*_{t,v}(x_s)\in U, \text{\ }z^*_{t,v}(x_s)\in V \text{\ }\forall (s,t,v)\in \Lambda^e(\ttt_{\omega^{\omega^\xi}}). \end{equation}  We note that $(2)$ holds after we make these replacements since for $s\prec t \prec u \preceq v\in MAX(\ttt_{\omega^{\omega^\xi}})$, $\theta_1(s)\prec \theta_1(t)\prec \theta_1(u)\preceq e_1(v)$, $e_1(v)\in MAX(\ttt_{\omega^{\omega^\xi}})$.  What we have done at this point is obtained vectors in $B_X$ and functionals in $K$ and $L$ so that the values $y^*_{t,v}(x_s)$ are nearly all equal for all $(s, t, v)\in \Lambda^e(\ttt_{\omega^{\omega^\xi}})$, the same holds for the values $z^*_{t,v}(x_s)$, without losing the property that the functionals $y^*_{t,v}+z^*_{t,v}$ still separate convex hulls of initial segments from tails.  The next step will be to make the values of $y^*_{t,v}(x_s)$ and $z^*_{t,v}(x_s)$ essentially constant when $t\prec s$ without losing the properties we have already obtained.

Next, define a function $f_2:\Lambda^e(\ttt_{\omega^{\omega^\xi}})\to \mathcal{P}\times \mathcal{P}$ by letting $f_2(t,u,v)=(U', V')$, where $(U', V')$ is the member of $\mathcal{P}\times \mathcal{P}$ such that $$(y^*_{t,v}(x_u), z^*_{t,v}(x_u))\in U'\times V'.$$  Using Corollary \ref{cut to the chase} in a way similar to the previous paragraph, we obtain an extended monotone map $(\theta_2, e_2)$ of $\ttt_{\omega^{\omega^\xi}}$ into itself and a pair $(U', V')\in \mathcal{P}\times \mathcal{P}$ so that for all $(t,u,v)\in \Lambda^e(\ttt_{\omega^{\omega^\xi}})$, $y^*_{\theta_2(t), e_2(v)}(x_{\theta_2(u)})\in U'$ and $z^*_{\theta_2(t), e_2(v)}(x_{\theta_2(u)})\in V'$.  Again, we may replace $x_u$ by $x_{\theta_2(u)}$, $y^*_{t,v}$ by $y^*_{\theta_2(t), e_2(v)}$, and $z^*_{t,v}$ by $z^*_{\theta_2(t), \theta_2(v)}$, note that (2) and (3) are still valid with these replacements, and assume \begin{equation} (y^*_{t,v}+z^*_{t,v})(x_s-x_u)\geqslant \ee \text{\ \ }\forall s\prec t \prec u\preceq v, v\in MAX(\ttt_{\omega^{\omega^\xi}}),\end{equation} \begin{equation} y^*_{t,v}(x_s)\in U, \text{\ }z^*_{t,v}(x_s)\in V \text{\ }\forall (s,t,v)\in \Lambda^e(\ttt_{\omega^{\omega^\xi}}), \end{equation}  \begin{equation} y^*_{t,v}(x_u)\in U', \text{\ }z^*_{t,v}(x_u)\in V' \text{\ }\forall (t,u,v)\in \Lambda^e(\ttt_{\omega^{\omega^\xi}}). \end{equation}

Let $\theta, \phi:\ttt_{\omega^{\omega^\xi}}\to \ttt_{\omega^{\omega^\xi}}$ be monotone maps satisfying the conclusion of Lemma \ref{shuffling}. Let $e$ be any extension of $\phi$.  Then for any $s\preceq t \prec u \preceq v$ with $v\in MAX(\ttt_{\omega^{\omega^\xi}})$, $\theta(s)\prec \phi(t) \prec \theta(u) \preceq e(v)$. This means that $(\theta(s), \phi(t), e(v)), (\phi(t), \theta(u), e(v))\in \Lambda^e(\ttt_{\omega^{\omega^\xi}})$.   Fix any $\lambda_1\in U$, $\lambda_2\in V$, $\mu_1\in U'$, and $\mu_2\in V'$.  Fix $s\preceq t \prec u\preceq v$, $v\in MAX(\ttt_{\omega^\xi})$.   Then using (4)-(6) and recalling that each of the four sets $U, V, U'$, and $V'$ has diameter not exceeding $\delta$, we deduce $$\ee \leqslant (y^*_{\phi(t), e(v)}+z^*_{\phi(t), e(v)})(x_{\theta(s)}- x_{\theta(u)}) \leqslant \lambda_1 - \mu_1 + \lambda_2 - \mu_2 +4\delta.$$  Thus we deduce $\max\{\lambda_1-\mu_1, \lambda_2-\mu_2\}\geqslant \ee/2 - 2\delta$.  Suppose $\lambda_1-\mu_1\geqslant \ee/2-2\delta$.  Then for arbitrary $s\preceq t \prec u \preceq v$, $v\in MAX(\ttt_{\omega^{\omega^\xi}})$, $$y^*_{\phi(t), e(v)}(x_{\theta(s)}- x_{\theta(u)}) \geqslant \lambda_1 -\mu_1 - 2\ee \geqslant \ee/2- 4\delta = \ee/2- \ee/6 = \ee/3.$$ Since $y^*_{\phi(t), e(v)}\in K$, this yields that $(x_{\theta(t|_i)})_{i=1}^{|t|}\in J(K, \ee/3)$ for all $t\in \ttt_{\omega^{\omega^\xi}}$.  Another appeal to \cite{C} yields that $j(K, \ee/3)>\omega^{\omega^\xi}$.

The second statement follows from the fact that with $K=A^*B_{Y^*}$, $L=B^*B_{Y^*}$, $(A+B)^*B_{Y^*}\subset K+L$.  Therefore $j((A+B)^*B_{Y^*}, \ee)>\omega^{\omega^\xi}$ for some $\ee>0$ implies that $j(K+L, \ee)>\omega^{\omega^\xi}$, and either $j(K, \ee/3)>\omega^{\omega^\xi}$ or $j(L, \ee/3)>\omega^{\omega^\xi}$.  

\end{proof}


\begin{thebibliography}{HD}

\normalsize
\baselineskip=17pt


\bibitem{BCFW} K. Beanland, R.M. Causey, D. Freeman, B. Wallis, \emph{Classes of operators determined by ordinal indices}, submitted.   

\bibitem{BF} K. Beanland, D. Freeman. \emph{Uniformly factoring weakly compact operators}, J. Funct. Anal. 266 (2014), no. 5, 2921-2943.

\bibitem{Bos} B. Bossard. \emph{Th\'{e}orie descriptive des ensembles en g\'{e}om\'{e}trie des espaces de Banach}. Th\'{e}se, Univ. Paris VI, (1994).  

\bibitem{B} J. Bourgain. On convergent sequences of continuous functions. \emph{Bull. Soc. Math. Belg. S\'{e}r. B}, 32(2):235-249, 1980. 

\bibitem{Brooker} P.A.H. Brooker, \emph{Asplund operators and the Szlenk index}, Operator Theory 68 (2012), 405-442.

\bibitem{C} R.M. Causey, \emph{Proximity to $\ell_p$ and $c_0$ in Banach spaces}, submitted.  

\bibitem{D} P. Dodos. \emph{Banach Spaces and Descriptive Set Theory: Selected Topics}, Lecture Notes in Mathematics, Vol. 1993, Springer, (2010).

\bibitem{James} R.C. James, \emph{Weak compactness and reflexivity}, Israel J. Math. 2 (1964), 101-119.  

\bibitem{JO} R. Judd, E. Odell. \emph{Concerning Bourgain's $\ell_1$ index of a Banach space}, Israel J. Math (1998)108:145-171. 

\bibitem{Ke} A. S. Kechris, \emph{Classical Descriptive Set Theory}, Grad. Texts in Math., 156, Springer-Verlag, 1995. 

\bibitem{KRN} K. Kuratowski, C. Ryll-Nardzewski. \emph{A general theorem on selectors}, Bull Acad. Pol. Sci. S\'{e}r. Sci., Math. Astr. et Phys. 13 (1965), 397-403. 

\bibitem{LT} D. H. Leung, W.K, Tang, \emph{The Bourgain $\ell^1$-index of mixed Tsirelson space}, J. Funct. Anal. 199 (2003), no. 2, 301-331.

\bibitem{M} J.D. Monk. \emph{Introduction to set theory}, McGraw-Hill, (1969). 

\bibitem{Ramsey} F. P. Ramsey, \emph{On a problem of formal logic}, Proc. London Math. Soc. (2) (1930), no. 30, 264-286.  

\end{thebibliography}
\end{document}